\documentclass[10pt]{amsart}
\usepackage{amsmath}
\usepackage{amssymb}
\usepackage{amsfonts}
\usepackage{float}
\usepackage{amsthm}\usepackage{graphicx,color}
\usepackage{thmtools}
\usepackage{thm-restate}

\usepackage{hyperref}
\newcounter{tmp}

\usepackage[all]{xy}
\UseComputerModernTips

\newtheorem*{theorem*}{Theorem}
\newtheorem{theorem}{Theorem}[section] 
\newtheorem{lemma}[theorem]{Lemma}     
\newtheorem{corollary}[theorem]{Corollary}
\newtheorem{proposition}[theorem]{Proposition}
\newtheorem*{conj*}{Question}
\newtheorem{conj}{Question}[theorem]
\theoremstyle{definition}
\newtheorem{definition}[theorem]{Definition}

\theoremstyle{remark}
\newtheorem{remark}[theorem]{Remark}
\numberwithin{equation}{section}

\theoremstyle{definition}

\newcommand{\C}{{\mathbb{C}}}

\newcommand{\PP}{{\mathbb{CP}}}
\newcommand{\R}{{\mathbb{R}}}
\newcommand{\Q}{{\mathbb{Q}}}

\newcommand{\M}{{\mathcal{M}}}

\newcommand{\In}{\mathfrak{i}}
\newcommand{\n}{\mathfrak{n}}

\newcommand{\w}{\widetilde{W}}
\newcommand{\0}{\textbf{0}}
\newcommand{\f}{\hat{f}}

\DeclareMathOperator{\ssp}{Spec}

\newcommand{\D}{\mathbb{D}}
\newcommand{\p}{\pi}

\DeclareMathOperator{\Reg}{Reg}
\DeclareMathOperator{\Sing}{Sing}

\DeclareMathOperator{\Id}{Id}
\DeclareMathOperator{\id}{Id}

\DeclareMathOperator{\rank}{rank}

\title[]
 {Periodic points of post-critically algebraic holomorphic endomorphisms.  } 

\author[ Van Tu Le]{Van Tu Le}
\address{Van Tu Le\\ Institut de Math\'ematiques de Toulouse, UMR5219\\ Universit\'e de Toulouse, CNRS\\ UPS IMT, 118 route de Narbonne, F-31062 Toulouse Cedex 9\\ France} 
\email{Van\_Tu.Le@math.univ-toulouse.fr}
\date{\today}
\thanks{2010 Mathematics Subject Classification: 32H50, 37F99.}
\thanks{\textit{Keywords: holomorphic dynamics, holomorphic endomorphisms, fixed points, periodic points, eigenvalues, multipliers.} }

\begin{document}
	\begin{abstract}A holomorphic endomorphism of $\PP^n$ is post-critically algebraic if its critical hypersurfaces are periodic or preperiodic. This notion generalizes the notion of post-critically finite rational maps in dimension one. We will study the eigenvalues of the differential of such a map along a periodic cycle. When $n=1$, a well-known fact is that the eigenvalue along a periodic cycle of a post-critically finite rational map is either superattracting or repelling. We prove that when $n=2$ the eigenvalues are still either superattracting or repelling. This is an improvement of a result by Mattias Jonsson. When $n\geq 2$ and the cycle is outside the post-critical set, we prove that the eigenvalues are repelling. This result improves one which was already obtained by Fornaess and Sibony under a hyperbolicity assumption on the complement of the post-critical set.
	\end{abstract}
	\maketitle

	\tableofcontents
	\section{Introduction}In this article, we will work with holomorphic endomorphisms of $\PP^n$. Without any further indication, every endomorphism considered in this article is holomorphic. Let ${f: \PP^n \rightarrow \PP^n}$ be an endomorphism. The critical locus $C_f$ is the set of points where the differential  
	${D_{z}f: T_{z}\PP^n \rightarrow T_{f(z)} \PP^n}$ is not surjective. The endomorphism $f$ is called \textit{post-critically algebraic}, PCA for short, if the post-critical set of $f$
	\[PC(f)=\bigcup\limits_{j\ge  1} f^{\circ j} (C_f)
	\]
	is an algebraic set of codimension one. If $n=1$, such an endomorphism is called \textit{post-critically finite} since proper algebraic sets in $\PP^1$ are finite sets. A point $z \in \PP^n$ is called a periodic point of $f$ of period $m$ if $f^{\circ m}(z) =z$ and $m$ is the smallest positive integer satisfying such a property. We define an eigenvalue of $f$ along the cycle of $z$ as an eigenvalue of $D_z f^{\circ m}$. We will study eigenvalues along periodic cycles of post-critically algebraic endomorphisms of $\PP^n$ of degree $d \ge 2$.
	
	When $n=1$, we have the following fundamental result:
	\begin{theorem}\label{dim one}
		Let $f:\PP^1\to \PP^1$ be a post-critically finite endomorphism of degree $d\ge 2$ and $\lambda$ be an eigenvalue of $f$ along a periodic cycle. Then either $\lambda =0$ or $|\lambda|>1$.
	\end{theorem}
	This theorem relies on the following relation of critical orbits and periodic points of endomorphisms of $\PP^1$. More precisely, without loss of generality, let $\lambda$ be the eigenvalue of an endomorphisms $f$ of $\PP^1$ (not necessary post-critically finite) at a fixed point $z$. Then,
	\begin{itemize}
		\item if $0 < |\lambda|<1$ or if $\lambda$ is a root of unity, then $z$ is the limit of the infinite orbit of some critical point,
		\item if $|\lambda|=1$ and $\lambda$ is not a root of unity then,
		\begin{itemize}
			\item either $z$ is accumulated by the infinite orbit of some critical point,
			\item or $z$ is contained in a Siegel disk whose boundary is accumulated by the infinite orbit of some critical point.
		\end{itemize}
	\end{itemize}
We refer to \cite{milnor2011dynamics} for further reading on this topic. In this article, we study how this result may be generalized to dynamics in dimension $n \ge 2$. We study the following question.
	
	\begingroup
	\setcounter{tmp}{\value{conj}}
	\setcounter{conj}{0} 
	\renewcommand\theconj{\Alph{conj}}
	\begin{conj}\label{conj 1}
		Let $f$ be a post-critically algebraic endomorphism of $\PP^n, n \ge 2$ of degree $d\geq 2$ and $\lambda$ be an eigenvalue of $f$ along a periodic cycle. Can we conclude that either $\lambda=0$ or $|\lambda|>1$?
	\end{conj}
	\endgroup
	\setcounter{theorem}{\thetmp} 
	In this article, we give an affirmative answer to this question when $n=2$.
	\begin{restatable}{thm}{PCA}\label{PCA P2}Let $f$ be a post-critically algebraic endomorphism of $\PP^2$ of degree $d\geq 2$ and $\lambda$ be an eigenvalue of $f$ along a periodic cycle. Then either $\lambda=0$ or $|\lambda|>1$.
	\end{restatable}
	Let us note that this question has been studied by several authors and some partial conclusions have been achieved. We refer to \cite{fornaess1992critically}, \cite{fornaess1994complex}, \cite{jonsson1998some}, \cite{ueda1998critical}, \cite{rong2008fatou}, \cite{gauthier2016symmetrization}, \cite{silvermann2019},\cite{gauthier2019geometric},\cite{ji2020structure}, \cite{astorg2018dynamics} for the study of post-critically algebraic endomorphisms in higher dimension. Concerning eigenvalues along periodic cycles, following Forn\ae ss-Sibony \cite[Theorem 6.1]{fornaess1994complex}, one can deduce that: for a PCA endomorphism $f$ of $\PP^n, n \ge 2$ such that $\PP^n \setminus PC(f)$ is Kobayashi hyperbolic and hyperbolically embedded, an eigenvalue $\lambda$ of $f$ along a periodic cycle outside $PC(f)$ has modulus at least or equal to $1$. Following Ueda \cite{ueda1998critical}, one can show that: the differential of a PCA endomorphism $f$ of $\PP^n, n \ge 2$ along a periodic cycle which is not critical has modulus at least $1$ (See Corollary \ref{cor_non attracting}  in this article). When $n=2$, Jonsson \cite{jonsson1998some} considered PCA endomorphisms of $\PP^2$ whose critical locus does not have periodic irreducible component. He proved that for such class of maps, every eigenvalue along a periodic cycle outside the critical set has modulus strictly bigger than $1$. Recently, in \cite{astorg2018dynamics}, Astorg studied the Question \ref{conj 1} under a mild transversality assumption on irreducible components of $PC(f)$.
	
	Our approach to prove Theorem \ref{PCA P2} is subdivided into two main cases: the cycle is either outside, or inside the post-critical set.
	
	When the cycle is outside the post-critical set, we improve the method of \cite{fornaess1994complex} to get rid of the Kobayashi hyperbolic assumption and exclude the possibility of eigenvalues of modulus $1$. We obtain the following general result.
	\begin{restatable}{thm}{firstform}\label{first case}
		Let $f$ be a post-critically algebraic endomorphism of $\PP^n$ of degree $d \ge 2$ and $\lambda$ be an eigenvalue of $f$ along a periodic cycle outside the post-critical set. Then $|\lambda|>1$.
	\end{restatable}
	When the cycle is inside the post-critical set, we restrict our study to dimension $n=2$. Let $z \in PC(f)$ be a periodic point of period $m$. Counting multiplicities, $D_z f^{\circ m}$ has two eigenvalues $\lambda_1$ and $\lambda_2$. We consider two subcases: either $z$ is a regular point of $PC(f)$, or $z$ is a singular point of $PC(f)$. 
	
	If the periodic point $z$ is a regular point of $PC(f)$, the tangent space $T_{z} PC(f)$ is invariant by $D_{z} f^{\circ m}$. Then $D_{z} f^{\circ m}$ admits an eigenvalue $\lambda_1$ with associated eigenvectors in  $T_{z} PC(f)$. The other eigenvalue $\lambda_2$ arises as the eigenvalue of the linear endomorphism $\overline{D_{z} f^{\circ m}} : T_{z} \PP^2 / T_{z} PC(f) \to T_{z} \PP^2 / T_{z} PC(f) $ induced by $D_{z} f^{\circ m}$. By using the normalization  of irreducible algebraic curves and Theorem \ref{dim one}, we prove that the eigenvalue $\lambda_1$ has modulus strictly bigger than $1$. Regarding the eigenvalue $\lambda_2$, following the idea used to prove Theorem \ref{first case}, we also deduce that either $\lambda_2 = 0$ or $|\lambda_2|>1$.
	
	If the periodic point $z$ is a singular point of $PC(f)$, in most of the cases, there exists a relation between $\lambda_1$ and $\lambda_2$. Then by using Theorem \ref{dim one}, we deduce that for $i =1,2$, either $\lambda_j = 0$ or $|\lambda_j|>1$. This has been already observed in \cite{jonsson1998some} but for the sake of completeness, we will recall the detailed statements and include the proof. 
	
	\textbf{Structure of the article:} In Section 2, we will recall the results of Ueda and prove that when a fixed point is not a critical point then every eigenvalue has modulus at least $1$. In Section 3, we will present the strategy and the proof of Theorem \ref{first case}. In Section 4, since the idea is the same as the proof of Theorem \ref{first case}, we will give the proof that for an eigenvalue $\lambda$ of a PCA map a periodic cycle which is a regular point of the post-critical set and the associated eigenvectors are not tangent to the post-critical set, then $\lambda$ is either zero or of modulus strictly bigger than $1$. In Section 5, we will study the dynamics of PCA endomorphisms of $\PP^2$ restricting on invariants curves and then prove Theorem \ref{PCA P2}. 
	
	\textbf{Acknowledgement:} The author is grateful to his supervisors Xavier Buff and Jasmin Raissy for their support,
	suggestions and encouragement. The author would like to thank also Matthieu Arfeux, Matthieu Astorg, Fabrizio 	Bianchi, Charles Favre, Thomas 	Gauthier,	Sarah 	Koch, Van Hoang Nguyen,
	Matteo 	Ruggiero,
	Johan	Taflin and
	Junyi	Xie for their comments and useful discussions. This work is supported by the fellowship of Centre International de Math\'ematiques et d'Informatique de Toulouse (CIMI).
	\\
	\textbf{Notations:} We denote by 
	
	\begin{itemize}
		\item[$\bullet$] $\D(0,r)=\{x \in \C|\|x\| < r \}$: the ball of radius $r$ in $\C$ (or simply $\D$ when $r=1$).
		\item[$\bullet$] $\D(0,R)^*$ or $\D^*$: the punctured disc $\D(0,R) \setminus \{0\}$.
		\item[$\bullet$] For two paths $\gamma,\eta: [0,1] \rightarrow X$ in a topological space $X$ such that $\gamma(1) = \eta(0)$, the concatenation path  
		$\gamma*\eta: [0,1] \rightarrow X$ is defined as
		\[\gamma* \eta(t) = \left\{\begin{array}{cc}
		\gamma(2t) \,, t \in [0,\frac{1}{2}]\\
		\eta(2t-1)\,, t \in [\frac{1}{2},1]
		\end{array}\right.
		.\]
		\item[$\bullet$]$\ssp(L)$: the set of eigenvalues of a linear endomorphism $L$ of a vector space $V$.
		\item[$\bullet$] For an algebraic set (analytic set) in a complex manifold $X$, we denote by $\Sing X$ the set of singular points of $X$ and by $\Reg X$ the set of regular points of $X$.
		
			\end{itemize}

\section{Periodic cycles outside the critical set}\label{sect_non critical fixed point}
In this section, we will prove that the eigenvalues of a post-critically algebraic endomorphism of $\PP^n$ at a fixed point, which is not a critical point, have modulus at least or equal to $1$. The proof relies on the existence of an open subset on which we can find a family of inverse branches and the fact that the family of inverse branches of endomorphisms of $\PP^n$ is normal. These results are due to Ueda, \cite{ueda1998critical}.  

Let us recall the definition of finite branched covering:
\begin{definition}
A proper, surjective continuous map $f: Y \rightarrow X$ of complex manifolds of the same dimension is called a \textit{finite branched covering} (or \textit{finite ramified covering}) if there exists an analytic set $D$ of codimension one in $X$ such that the map
\[f: Y\setminus f^{-1}(D) \rightarrow X \setminus D
\] 
is a covering. We say that $f$ is \textit{ramified over $D$} or a \textit{$D$-branched covering}. The set $D$ is called the \textit{ramification locus.}
\end{definition}
We refer to \cite{gunning1990introduction} for more information about the theory of finite branched coverings. 

Recall that endomorphisms $f$ of $\PP^n$ of degree $d \ge 2$ are finite branched covering ramifying over $f(C_f)$. If $f$ is post-critically algebraic, for every $j \ge 1$, $f^{\circ j}$ is ramified over $PC(f)$.

Let $z \notin C_f$ be a fixed point of $f$. Then $f^{\circ j}$ is locally invertible in a neighborhood of $z$. The following result which are due to Ueda ensures that we can find a common open neighborhood on which inverse branches of $f^{\circ j}$ fixing $z$ are well-defined for every $j \ge 1$.
\begin{lemma}[\cite{ueda1998critical}, Lemma 3.8]\label{lm_normal neighborhood}
	Let $X$ be a complex manifold and $D$ be an analytic subset of $X$ of codimension $1$. For every point $x \in X$, if $W$ is a simply connected open neighborhood of $x$ such that $(W,W \cap D)$ is homeomorphic to a cone over $(\partial W, \partial W \cap D)$ with vertex at $x$, then for every branched covering $\eta: Y \rightarrow W$ ramifying over $D \cap W$, the set $\eta^{-1}(x)$ consists of only one point.
		\footnote{The important consequence of this hypothesis is that, for every open neighborhood of $W_0$ of $z$ in $W$, the homomorphism $\pi_1(W_0 \setminus D) \to \pi_1(W \setminus D)$ is surjective }
\end{lemma}
Given a topological space $T$, the cone over a set $K \subset T$ with vertex at a point $x \in T$ is the quotient space $Cone(K) = K \times [0,1]/ K \times \{0\}$ where $x$ is identified with the equivalent class of $K \times \{0\}$, which is a point in $Cone(K)$.
\begin{proposition}\label{prop_inverse branches}
	Let $f$ be a post-critically algebraic endomorphism of $\PP^n$ of degree $d \ge 2$ and let $z \notin C_f$ be a fixed point of $f$. Let $W$ be a simply connected open neighborhood of $z$ such that $(W,W \cap PC(f))$ is homeomorphic to a cone over $(\partial W, \partial W \cap PC(f))$ with vertex at $z$. Then there exists a family of holomorphic inverse branches $h_j: W \to \PP^n$ of $f^{\circ j}$ fixing $z$, that is,
	\[h_j (z) =z, f^{\circ j} \circ h_j= \Id|_W.
	\]
\end{proposition}
Note that for a fixed point $z$ of a post-critically algebraic endomorphism $f$ of $\PP^n$, since $PC(f)$ is an algebraic set, there always exists a simply connected neighborhood $W$ of $z$ such that ${(W,W \cap PC(f))}$ is homeomorphic to a cone over $(\partial W, \partial W \cap PC(f))$ with vertex at $z$. Indeed, if $z \notin PC(f)$ then we can take any simply connected neighborhood $W$ of $z$ in $\PP^n \setminus PC(f)$. If $z \in PC(f)$ then it follows from \cite[Theorem 3.2]{seade2019milnor} that such a neighborhood always exists. We refer to \cite{milnor2016singular} to an approach when $z$ is an isolated singularity of $PC(f)$, see also \cite[Remark 2.3]{seade2005topology}.
\begin{proof}[Proof of Proposition \ref{prop_inverse branches}]
	For every $j \ge 1$, denote by $W_j$ the connected component of $f^{-j}(W)$ containing $z$. Since $f^{\circ j}$ are branched covering ramifying over $PC(f)$, $f^{\circ j}$ induces a branched covering
	\[f_j:=f^{\circ j}|_{W_j}: W_j \to W
	\]
	ramifying over $W \cap PC(f)$. By Lemma \ref{lm_normal neighborhood}, we deduce that $f_j^{-1}(z)$ consists of only one point, which is in fact $z$. Since $W$ is simply connected hence connected, the order of the branched covering $f_j$ coincides with the branching order of $f^{\circ j}$ at $z$. Note that $z$ is not a critical point of $f^{\circ j}$ therefore the branching order of $f^{\circ j}$ at $z$ is $1$. This means that $f_j$ is branched covering of order one of complex manifolds, thus $f_j$ is homeomorphism hence biholomorphism (see \cite[Corollary 11Q]{gunning1990introduction}). The holomorphic map $h_j: W \to W_j$ defined as the inverse of $f_j$ is the map we are looking for.
\end{proof}
Once we obtain a family of inverse branches, the following theorem, which is due to Ueda, implies that this family is normal.
\begin{theorem}[{\cite{ueda1998critical},Theorem 2.1}]\label{thm_normality of inverse branches}
	Let $f$ be an endomorphism of $\PP^n$. Let $X$ be a complex manifold with a holomorphic map $\pi: X \to \PP^n$. Let $\{h_j: X \to \PP^n\}_j$ be a family of holomorphic lifts of $f^{\circ j}$ by $\pi$, that is $ f^{\circ j} \circ h_j = \pi$. Then $\{h_j\}_j$ is a normal family.
\end{theorem}
Thus, for a fixed point $z$ of a post-critically algebraic endomorphism $f$, if $z$ is not a critical point (or equivalently, $D_z f$ is invertible), we can obtain an open neighborhood $W$ of $z$ in $\PP^n$ and a normal family of holomorphic maps $\{h_j:W \to \PP^n\}_j$ such that 
\[f^{\circ j} \circ h_j = \Id_W, h_j(z) =z.
\]
The normality of $\{h_j\}_j$ implies that $\{D_z h_j\}_j$ is a uniformly bounded sequence (with respect to a fixed norm $\|\cdot\|$ on $T_{z} \PP^n$). Since ${D_z h_j = \left(D_z f\right)^{-j}}$, we can deduce that every eigenvalue of $D_{z} f$ has modulus at least $1$.  Consequently, we can find a $D_z f$-invariant decomposition of $T_z \PP^n$ as the following direct sum
\[T_{z} \PP^n = \left(\bigoplus\limits_{\lambda \in \ssp D_z f, |\lambda|=1} E_\lambda\right) \oplus  \left(\bigoplus_{\lambda \in \ssp D_z f, |\lambda |>1} E_\lambda\right)=E_n \oplus E_r
\]
where $E_\lambda$  is the generalized eigenspace of the eigenvalue $\lambda$. We call $E_n$ the neutral eigenspace and $E_r$ the repelling eigenspace of $D_{z}f$ (see also \ref{sect_strategy}). If $E_\lambda$ is not generated by eigenvectors (or equivalently, $D_z f|_{E_\lambda}$ is not diagonalizable), we can find at least two generalized eigenvectors $e_1,e_2$ of $D_z f$ corresponding to $\lambda$ such that 
\[D_z f (e_1) = \lambda e_1, D_z f(e_2) = \lambda e_2+ e_1.
\]
Then $D_z h_j(e_2) = \lambda^{-j}e_2 -j\lambda^{-(j+1)}e_1$. If $|\lambda|=1$, then $\|D_z h_j(e_2)\|$ tends to $\infty$ as $j$ tends to $\infty.$ This contradicts the uniformly boundedness of $\{D_z h_j\}_j$. Hence $D_z f|_{E_\lambda}$ is diagonalizable. Thus, we have proved the following result.
\begin{corollary}\label{cor_non attracting}Let $f$ be a post-critically algebraic endomorphism of $\PP^n$ of degree $d\ge 2$ and let $\lambda$ be an eigenvalue of $f$ at a fixed point $z \notin C_f$. Then $|\lambda| \ge 1$. Moreover, if $|\lambda|=1$, $D_{z} f|_{E_{\lambda}}$ is diagonalizable.
\end{corollary}

\section{Periodic cycles outside the post-critical set}\label{sect_outside the post-critical set}
In this section, we prove Theorem \ref{first case}. 
\firstform* 
Observe that a periodic point of $f$ is simply a fixed point of some iterate of $f$. Moreover, any iterate of a PCA maps is still PCA. Thus, it is enough to proof Theorem \ref{first case} when $\lambda$ is an eigenvalue of $f$ at a fixed point $z \notin PC(f)$. We will consider an equivalent statement and prove this equivalent statement. 
\subsection{Equivalent problem in the affine case.}\label{sec: equivalent problem}
Recall that for an endomorphism $f: \PP^n \rightarrow \PP^n$ of degree $d$, there exists a polynomial endomorphism \[F=(P_1,\ldots,P_{n+1}): \C^{n+1} \to \C^{n+1}\] where $P_i$ are homogeneous polynomials of the same degree $d \ge 1$ and $F^{-1}(0)=\{0\}$ such that 
\[ f \circ \pi = \pi \circ F
\] 
where $\pi: \C^{n+1} \setminus \{0\}  \rightarrow \PP^n$ is the canonical projection. The integer $d$ is called the \textit{algebraic degree} (or \textit{degree}) of $f$. Such a map $F$ is called a \textit{lift} of $f$ to $\C^{n+1}$. Further details about holomorphic endomorphisms of $\PP^n$ and their dynamics can be found in \cite{sibony1999dynamique}, \cite{gunning1990introduction}, \cite{DS08}, \cite{fornaess1994complex},\cite{fornaess1995complex}.

Lifts to $\C^{n+1}$ of an endomorphism of $\PP^n$ belong to a class of \textit{non-degenerate homogeneous polynomial endomorphisms} of $\C^{n+1}$. More precisely, a non-degenerate homogeneous polynomial endomorphism of $\C^n$ of algebraic degree $d$ is a polynomial map $F: \C^n \rightarrow \C^n$ such that $F(\lambda z) = \lambda^d z$ for every $z \in \C^n, \lambda \in \C$ and $F^{-1}(0)=\{0\}$. Conversely, such a map induces an endomorphism of $\PP^n$. This kind of maps has been studied extensively in \cite{hubbard1994superattractive}. If we consider $\C^{n+1}$ as a dense open set of $\PP^{n+1}$ by the inclusion $(\zeta_0,\zeta_1,\ldots,\zeta_n) \mapsto [\zeta_0:\zeta_1:\ldots:\zeta_n:1]$ then $F$ can be extended to an endomorphism of $\PP^{n+1}$. Moreover, this extension fixes the hypersurface at infinity $\PP^{n+1} \setminus \C^{n+1} \cong \PP^n$ and the restriction to this hypersurface is the endomorphism of $\PP^n$ induced by $F$.

Thus, if $f$ is a post-critically algebraic endomorphism of $\PP^n$, every lift $F$ of $f$ to $\C^{n+1}$ is the restriction of a post-critically algebraic endomorphism of $\PP^{n+1}$ to $\C^{n+1}$. Post-critically algebraic non-degenerate homogeneous polynomial endomorphisms of $\C^{n+1}$ have similar properties, which are proved in Section 2, as post-critically algebraic endomorphisms of $\PP^n$. More precisely, we can sum up in the following proposition.
\begin{proposition}\label{prop_properties of affine homogeneous maps}
	Let $F$ be a post-critically algebraic non-degenerate homogeneous polynomial endomorphism of $\C^{n+1}$ of degree $d \ge 2$ and let $z \notin C_F$ be a fixed point of $F$. Then, we have the following assertions.
	\begin{itemize}
		\item[a.] Let $X$ be a complex manifold and let $\p: X \to \C^{n+1}$ be a holomorphic map. Then every family of holomorphic maps $\{h_j: X \to \C^{n+1}\}_j$, which satisfies that $F^{\circ j } \circ h_j = \p$ for every $j \ge 1$, is normal.
		\item[b.] There exist a simply connected open neighborhood $W$ of $z$ in $\C^{n+1}$ and a family $\{h_j: W \to \C^{n+1}\}_j$ of inverse branches of iterates of $F$, i.e. $F^{\circ j} \circ h_j = \Id_W$, fixing $z$.
		\item[c.]  Every eigenvalue $\lambda$ of $\ssp(D_z F)$ has modulus at least $1$. The tangent space $T_z \C^{n+1}$ admits a $D_z F$-invariant decomposition $T_z \C^{n+1} = E_n \oplus E_r$ where the neutral eigenspace $E_n$ is the sum of generalized eigenspaces corresponding to eigenvalues of modulus $1$ and the repelling $E_r$ is the sum of generalized eigenspaces corresponding to eigenvalues of modulus strictly bigger than $1$.
		\item[d.] If $|\lambda|=1$ then $D_z f|_{E_\lambda}$ is diagonalizable.
	\end{itemize}
\end{proposition}

\begin{remark}\label{Remark} Regarding eigenvalues at fixed points, studying eigenvalues at fixed points of post-critically algebraic endomorphisms of $\PP^n$ is equivalent to studying eigenvalues at fixed points of post-critically algebraic non-degenerate homogeneous polynomial self-map of $\C^{n+1}$. More precisely, let $f : \PP^n \to \PP^n$ be an endomorphism of degree $d \ge 2$ and let $F$ be a lift of $f$. Assume $z$ is a fixed point of $f$. Then, the complex line $L$ containing $\p^{-1}(z)$ is invariant under $F$ and the map induced by $F$ on $L$ is conjugate to $x \mapsto x^d$. In particular, there exists a fixed point $w \in L \setminus \{0\}$ of $F$ such that $\p(w)=z$ and $D_{w}F$ preserves $T_{w} L \subset T_{w} \C^{n+1}$ with an eigenvalue $d$. Then, $D_{w} F$ descends to a linear endomorphism of the quotient space $T_{w} \C^{n+1}/T_{w} (\C w)$ which is conjugate to $D_{z}f : T_{z} \PP^n \to T_{z} \PP^n$. Hence a value $\lambda$ is an eigenvalue of $D_{w}F$ if and only if either $\lambda$ is an eigenvalue of $D_{z}f$ or $\lambda = d$.
	
	Conversely, if $w$ is a fixed point of $F$ then either $w=0 \in \C^{n+1}$ (and the eigenvalues of $D_0F$ are all equal to $0$) or $w$ induces a fixed point $\p(w)$ of $f$. Since we consider only post-critically algebraic endomorphisms of degree $d \ge 2$, Question \ref{conj 1} is equivalent to the following question.
\end{remark}
\begingroup
\setcounter{tmp}{\value{conj}}
\setcounter{conj}{1} 
\renewcommand\theconj{\Alph{conj}}
\begin{conj}
	Let $f$ be a non-degenerate homogeneous polynomial post-critically algebraic endomorphism of $\C^n$ and $\lambda$ be an eigenvalue of $f$ along a periodic cycle. Then either $\lambda=0$ or $|\lambda|>1$.
\end{conj}
\endgroup
\setcounter{theorem}{\thetmp} 
The advantage of this observation is that we can make use of some nice properties of the affine space $\C^n$. More precisely, we will use the fact that the tangent bundle of $\C^n$ is trivial
\subsection{Strategy of the proof of Theorem \ref{first case}}\label{sect_strategy} Recall that an eigenvalues $\lambda$ of $D_{z} f: T_{z} \C^n \to T_{z} \C^n$ is called
\begin{itemize}
	\item superattracting if $\lambda=0$,
	\item attracting if $0<|\lambda|<1$,
	\item neutral if $|\lambda|=1$,
	\begin{itemize}
		\item parabolic or rational if $\lambda$ is a root of unity,
		\item elliptic or irrational if $\lambda$ is not a root of unity,
	\end{itemize}
	\item repelling if $|\lambda|>1$,
\end{itemize}
By Remark \ref{Remark}, in order to prove Theorem \ref{first case}, it is enough to prove the following result.
	\begin{restatable}{thm}{affine case }\label{affine case}
	Let $f$ be a post-critically algebraic non-degenerate homogeneous polynomial endomorphism of $\C^n$ of degree ${d \ge 2}$ and $\lambda$ be an eigenvalue of $f$ at a fixed point $z \notin PC(f)$. Then $\lambda$ is repelling.
\end{restatable}
The strategy of the proof is as follows.
\begin{itemize}
	\item[Step 1.]Set $X= \C^n \setminus PC(f)$ and let $\p: \widetilde{X} \rightarrow X$\footnote{Since we are now only work on $\C^n$, we won't use $\pi$ as the canonical projection from $\C^{n+1}$ to $\PP^n$} be its universal covering. We construct a holomorphic map $g: \widetilde{X} \rightarrow \widetilde{X} $ such that 	
	\[  f \circ \p \circ g =\p
	\]
	and $g$ fixes a point $[z]$ such that $\p([z]) = z$. 
	\item[Step 2.] We prove that the family $\{g^{\circ m}\}_m$ is normal. Then there exists a closed complex submanifold $M$ of $\widetilde{X}$ passing through $[z]$ such that $g|_M$ is an automorphism and $\dim M$ is the number of eigenvalues of $D_{z} f$  of modulus $1$, i.e. neutral eigenvalues, counted with multiplicities. Due to Corollary \ref{cor_non attracting}, it is enough to prove that $\dim M = 0$.
	\item[Step 3.] In order to prove that $\dim M =0$, we proceed by contradiction. Assuming that $\dim M >  0$. We then construct a holomorphic mapping $\Phi: M \rightarrow T_{[z]} M $ such that  $\Phi([z]) = 0$, $ D_{[z]} \Phi = \id$ and
	\[ \Phi \circ g = D_{[z]} g \circ \Phi.
	\]
	We deduce that $D_{z} f$ has no parabolic eigenvalue.
	\item[Step 4.] Assume that $\lambda$ is a neutral irrational eigenvalue and $v$ is an associated eigenvector. We prove that the irreducible component $\Gamma$ of $\Phi^{-1}(\C{v})$ containing $[z]$ is smooth and $\Phi|_{\Gamma}$ maps $\Gamma$ biholomorphically onto a disc $\D(0,R)$ with $0<R < +\infty$. 
	\item[Step 5.] Denote by $\kappa: \D(0,R) \rightarrow M$ the inverse of $\Phi|_{\Gamma}$.  We prove that $\p \circ \kappa$ has radial limits almost everywhere on $\partial \D(0,R)$ and these radial limits land on $\partial X$. We deduce from this a contradiction and Theorem \ref{affine case} is proved.
\end{itemize} 
\subsection{Lifting the backward dynamics via the universal covering} Denote by $X= \C^n \setminus PC(f)$ the complement of $PC(f)$ in $\C^n$. Since $PC(f)$ is an algebraic set, the set $X$ is a connected open subset of $\C^n$ then the universal covering of $X$ is well defined. Denote by $\pi:\widetilde X\to X$ the universal covering of $X$ defined by 
\[\widetilde X  = \bigl\{ [\gamma]~|~\gamma\text{ is a path in }X~\text{ starting at }z_0\bigr\}\quad \text{and}\quad \pi\bigl([\gamma]\bigr) = \gamma(1),\]
where $[\gamma]$ denotes the homotopy class of $\gamma$ in $X$, fixing the endpoints $\gamma(0)$ and $\gamma(1)$. Denote by $[z]$ the element in $\widetilde{X}$ representing the homotopy class of the constant path at $z$. We endow $\widetilde X$ with a complex structure such that $\pi:\widetilde X\to X$ is a holomorphic covering map.

Set $Y = f^{-1}(X)\subset X$ and $\widetilde Y = \pi^{-1}(Y)\subset \widetilde X$. Since $f:Y\to X$ is a covering map, every path $\gamma$ in $X$ starting at $z_0$ lifts to a path $f^*\gamma \subset Y$ starting at $z_0$. In addition, if $\gamma_1$ and $\gamma_2$ are homotopic in $X$, then $f^*\gamma_1$ and $f^*\gamma_2$ are homotopic in $Y\subset X$, in particular in $X$. Thus, this pullback map $f^*$ induces a map $g:\widetilde X\to \widetilde X$ such that the following diagram commutes: 
\[\xymatrix{ 
	\widetilde{X}\ar[d]_{\p}	&\widetilde{X} \ar[l]_{g} \ar[d]^{\p}\\
	X \ar[r]_{f} & X}
\] 
Note that $g([z])=[z]$. In addition, $g$ is holomorphic since in local charts given by $\pi$, it coincides with inverse branches of $f$. 
\subsection{Normality of maps on the universal covering}\label{sec_step 2 case 1} We will prove that the family $\{g^{\circ j}\}_j$ is a normal family.
For every integer $j  \ge 1$, define $k_j = \p \circ g^{\circ j}$ so that $f^{\circ j} \circ k_j =\p$.
\begin{lemma}\label{normal}
	The family	$\{k_j: \widetilde{X} \rightarrow X\}_j$ is normal and any limit takes values in $X$.
\end{lemma}
\begin{proof}
	By Proposition \ref{prop_properties of affine homogeneous maps}, the family $\{k_j : \widetilde{X} \rightarrow \C^n \}_j$ is normal. Denote by $Q: \C^n \rightarrow \C$ a polynomial such that $PC(f)$ is the zeros locus of $Q$. Consider the family 
	\[Q_j= Q\circ k_j : \widetilde{X} \rightarrow \C.
	\]
	Since $k_j(\widetilde{X}) \subset X$, the family $\{Q_j\}$ is a normal family of nonvanishing functions. Then by Hurwitz's theorem, every limit map is either a nonvanishing function or a constant function. But $Q_j([z]) = Q(z) \neq 0$ hence every limit map is a nonvanishing function, i.e. every limit map of $\{k_j\}$ is valued in $X$. Thus $\{k_j: \widetilde{X}\rightarrow X\}$ is normal.
\end{proof}
We can deduce the normality of $\{g^{\circ j}\}_j$.
\begin{proposition}\label{normality of g}
	The family $\{g^{\circ j}\}_j$ is normal.
\end{proposition}
\begin{proof}
	Let $\{g^{\circ j_s}\}_s$ be a sequence of iterates of $g$. Extracting subsequences if necessary, we can assume that $k_{j_s}$ converges to a holomorphic map $k: \widetilde{X} \rightarrow X$.
	
	Since $\widetilde X$ is simply connected and $\pi:\widetilde X\to X$ is a holomorphic covering map, there exists a holomorphic map $g_0:\widetilde X\to \widetilde X$ such that $\pi\circ g_0 = k$ and $g_0([z]) = [z]$. 
	Note that for every $j\geq 1$, $g^{\circ j}([z]) = [z]$, thus the sequence $\bigl\{g^{\circ j}([z])\bigr\}_j$ converges to $g_0([z])$. According to \cite[Theorem 4]{andreian2003coverings}, the sequence $\bigl\{g^{\circ j_s}\bigr\}_s$ converges locally uniformly to $g_0$. This shows that $\bigl\{g^{\circ j}\bigr\}_j$ is normal. \end{proof}
\begin{remark}
	The proof of Lemma \ref{normal} relies on the post-critically algebraic hypothesis. Without the post-critically algebraic assumption, for a fixed point $z$ which is not accumulated by the critical set, we can still consider the connected component $U$ of $\C^n \setminus \overline{PC(f)}$ containing $z$ and the construction follows. Then we will need some control on the geometry of $U$ to prove that the family $\{g^{\circ j}\}$ is normal. For example, if $U$ is a pseudoconvex open subset of $\C^n$ or in general, if $U$ is a taut manifold, then $\{g^{\circ j}\}_j$ is normal.
\end{remark}

\subsection{Consequences of normality}\label{sec: step 2} The normality of the family of iterates of $g$ implies many useful information. In particular, following Abate \cite[Corollary 2.1.30-2.1.31]{abate1989iteration}, we derive the existence of a center manifold of $g$ on $\tilde{X}$.
\begin{theorem}\label{normality consequences}
	Let $X$ be a connected complex manifold and let $g$ be an endomorphism of $X$. Assume that $g$ has a fixed point $z$. If the family of iterates of $g$ is normal, then
	\begin{enumerate}
		\item Every eigenvalue of $D_{z} g$ is contained in the closed unit disc.
		\item The tangent space $T_{z} X$ admits a $D_{z} g$-invariant decomposition $T_{z} X  = \widetilde{ E_n } \oplus \widetilde{E_a}$ such that $D_{z} g |_{\widetilde{ E_n }}$ has only neutral eigenvalues and $D_{z} g|_{\widetilde{E_a}}$ has only attracting or superattracting eigenvalues.
		\item The linear map $D_{z} g|_{\widetilde{E_n}}$ is diagonalizable. 
		\item There exists a limit map $\rho$ of iterates of $g$ such that $\rho \circ \rho = \rho$
		\item The set of fixed points of $\rho$, which is $\rho(X)$, is a closed submanifold of $X$. Set $M = \rho(X)$.
		\item The submanifold $M$ is invariant by $g$. In fact, $g|_M$ is an automorphism.
		\item The submanifold $M$ contains $z$ and $T_{z} M = \widetilde{E_n}$.
	\end{enumerate}
\end{theorem} 
Applying this theorem to $\widetilde{X}$ and $g: \widetilde{X} \rightarrow \widetilde{X}$ fixing $[z]$, we deduce that $D_{[z]}g$ has only eigenvalues of modulus at most $1$ and $T_{[z]} \widetilde{X}$ admits a $D_{[z]}g$-invariant decomposition as $T_{[z]} \widetilde{X} = \widetilde{E_n} \oplus \widetilde{E_a}$. Differentiating both sides of $f \circ \p \circ g = \p$ at $[z]$, we have
\[D_{z} f \cdot D_{[z]} \p \cdot D_{[z]} g =D_{[z]} \p.
\]
Hence $\lambda$ is a neutral eigenvalue of $D_{z}f$ if and only if $\lambda^{-1}$ is a neutral eigenvalue of $D_{[z]} g$. Consequently, $D_{[z]} \p$ maps $\widetilde{E_n}$ to the neutral eigenspace $E_n$ of $D_z f$, $\widetilde{E_a}$ onto the repelling eigenspace $E_r$ of $D_z f$ (see Proposition \ref{prop_properties of affine homogeneous maps}).

We also obtain a closed center manifold $M$ of $g$ at $[z]$, i.e. if $\lambda$ is a neutral eigenvalue of $D_{[z]} g$ of eigenvector $v$, then $v \in T_{[z]} M$. So in order to prove Theorem \ref{affine form}, it is enough to prove that $\dim M=0$. The first remarkable property of $M$ is that $\p(M)$ is a bounded set in $\C^n$.
\begin{proposition}\label{backward manifold}The image $\p(M)$ is bounded.
\end{proposition}
\begin{proof} For every $[\gamma] \in \widetilde{X}$, note that $\{\pi(g^{\circ j}([\gamma]))\}_j$ is in fact a sequence of backward iterations of $\gamma(1)$ by $f$ and the $\omega$-limit set of backward images of $\C^n$ by $f$ is bounded. More precisely, since $f$ is a homogeneous polynomial endomorphism of $\C^n$ of degree $d \ge 2$, the origin $0$ is superattracting and the basin of attraction $\mathcal{B}$ is bounded with the boundary is $\partial \mathcal{B}=H_f^{-1}(0)$ where
	\[H_f(w)=\lim\limits_{j \rightarrow \infty} \frac{1}{d^j}\log\|f^{\circ j} (w) \|
	\]
	for $w \in \C^n \setminus \{0\}$. The function $H_f$ is called the \textit{potential function} of $f$ (see \cite{hubbard1994superattractive}). It is straight forward by computation that $H_f(\pi(m)) =0$ for every $m \in M$. Hence $\pi(M) \subset \partial \mathcal{B}$ is bounded.
\end{proof}
\subsection{Semiconjugacy on the center manifold}\label{sec_semiconjugacy} Assume that $\dim M>0$. Denote by $\Lambda$ the restriction of $D_{[z]} g$ on ${T_{[z]} M}$. The following proposition assures that we can semiconjugate $g|_M$ to $\Lambda$.
\begin{proposition}\label{linearizing}
	Let $M$ be a complex manifold and let $g$ be an endomorphism of $M$ such that the family of iterates of $g$ is normal. Assume that $g$ has a fixed point $z$ such that $D_{z} g$ is diagonalizable with only neutral eigenvalues and that there exists a holomorphic map $\varphi: M \rightarrow T_{z} M$ such that $\varphi(z) = 0,D_{z} \varphi = \id$. Then there exists a holomorphic map $\Phi: M \rightarrow T_{z} M$ such that $\Phi(z) = 0, D_{z} \Phi =\id$ and 
	\[D_{z} g \circ \Phi = \Phi \circ g.
	\]
\end{proposition}
\begin{proof}
	Consider the family $\{ (D_{z} g)^{-n}\circ \varphi \circ g^{\circ n } : M \rightarrow T_{[z]}M \}_n$, we know that $\{g^{\circ n} \}_n$ is normal, thus $\{\varphi \circ g^{\circ j}\}_j $ is locally uniformly bounded. The linear map $D_{z} g$ is diagonalizable with neutral eigenvalues, so $\{D_{z} g^{-n}\}_n$ is uniformly bounded on any bounded set. Then ${\{ (D_{z} g)^{-n}\circ \varphi \circ g^{\circ n} \}_n}$ is a normal family. Denote by $\Phi_N$ the Cesaro average of ${\{ (D_{z} g)^{-n}\circ \varphi \circ g^{\circ n} \}_n}$.
	\[\Phi_N = \frac{1}{N} \sum\limits_{n=0}^{N-1} (D_{z} g)^{-n} \circ \varphi \circ g^{\circ n} .\]
	The family $\{\Phi_N\}_N$ is also locally uniforly bounded, thus normal.  Observe that 
	\begin{align*}
	\Phi_N \circ g &= \frac{1}{N} \sum\limits_{n=0}^{N-1} (D_{z} g)^{-n} \circ \varphi \circ g^{\circ (n+1)}\\
	&=D_{z} g \Phi_N + D_{z} g\left(-\frac{1}{N} \varphi+ \frac{1}{N} \left( (D_{z} g)^{-(N+1)} \circ \varphi \circ g^{\circ (N+1)}\right) \right).
	\end{align*}
	For every subsequence $\{N_k\}$, the second term on the right hand side converges locally uniformly to $0$. So for every limit map $\Phi$ of $\{\Phi_N \}_N$, $\Phi$ satisfies that 
	\[
	\Phi \circ g = D_{z} g \circ \Phi.
	\]
	Since $g$ fixes $z$ we have that for every $N \ge 1$,
	\begin{align*}
	D_{z} \Phi_N &=  \frac{1}{N} \sum\limits_{n=0}^{N-1} (D_{z} g)^{-n} \circ D_{z} \varphi \circ D_{z}g^{\circ n}=\Id.\\
	\end{align*}
	So $D_{z} \Phi= \Id$ for every limit map $\Phi$ of $\{\Phi_N\}_N$.
\end{proof}   
Now we consider the complex manifold $M$ obtained in Step \ref{sec: step 2} and the restriction of $g$ on $M$ which is an automorphism with a fixed point $[z]$. Since $\Lambda$ has only neutral eigenvalues, in order to apply Proposition \ref{linearizing}, we need to construct a holomorphic map $\varphi: M \rightarrow T_{[z]} M$ such that $D_{[z]} \varphi = \id$ . The map $\varphi$ is constructed as the following composition:
\[M \xrightarrow{\In} \widetilde{X} \xrightarrow{\p} X \xrightarrow{\delta} T_{z} X \xrightarrow{ (D_{[z]}\p)^{-1}} T_{[z]} \widetilde{X} \xrightarrow{\pi_{\widetilde{E_a}}} T_{[z]} M.
\]
where $\delta: X \rightarrow T_{z} X$ is a holomorphic map tangent to identity, $\pi_{\widetilde{E_a}}: T_{[z]} \widetilde{ X } \rightarrow T_{[z]} M$ is the projection parallel to $\widetilde{E_a}$, $\In: M \rightarrow \widetilde{X}$ is the canonical inclusion and its derivative $D_{[z]} \In : T_{[z]} M \rightarrow T_{[z]} \widetilde{X}$ is again the canonical inclusion. Then $D_{[z]} \varphi: T_{[z]} M \rightarrow T_{[z]} M$ is 
\begin{align*}
D_{[z]} \varphi &= D_{[z]}(\pi_{\widetilde{E_a}} \circ (D_{[z]} \p)^{-1} \circ \delta\circ \p \circ \In )\\
&= \pi_{\widetilde{E_a}} \circ (D_{[z]} \pi)^{-1}  \circ D_{z} \delta \circ D_{[z]} \p \circ D_{[z]} \In =\Id.\\
\end{align*}
\begin{remark}
	The existence of a holomorphic map $\delta: X \rightarrow T_{z} X$ tangent to identity is one of the advantages we mentioned in Remark \ref{Remark}. It comes from the intrinsic nature of the tangent space of affine spaces. In this case, $X$ is an open subset of $\C^n$ which is an affine space directed by $\C^n$.
\end{remark}
\begin{corollary}\label{cor_neutral is irrational}
	Let $z$ be a fixed point of a non-degenerate homogeneous polynomial post-critically algebraic endomorphism $f$ of $\C^n$. Assume that ${z \notin PC(f)}$. If $\lambda$ is a neutral eigenvalue of $D_{z} f$ then $\lambda$ is an irrational eigenvalue. 
\end{corollary}
\begin{proof}
	It is equivalent to consider a neutral eigenvalue $\lambda$  of $D_{[z]} g$ and assume that $\lambda=e^{2 \pi i \frac{p}{q}}$. Hence $(D_{[z]} g)^{q}$ fixes pointwise the line $\C v$ in $T_{[z]} M$. This means that locally near $[z]$, $g^{\circ q}$ fixes $\Phi^{-1}(\C v)$ hence $f^{\circ q}$ fixes $\p(\Phi^{-1}(\C v))$ near $z$. Note that $\Phi$ is locally invertible near $[z]$ hence $\Phi^{-1}(\C v)$ is a complex manifold of dimension one near $[z]$. Then $\p(\Phi^{-1}(\C v))$ is a complex manifold near $z$ because $\p$ is locally biholomorphic. In particular, $\p(\Phi^{-1}(\C v))$ contains uncountably many fixed points of $f^{\circ q}$. This is a contradiction because $f^{\circ q}$ has only finitely many fixed points (see \cite[Proposition 1.3]{DS08}). Hence $\lambda$ is an irrational eigenvalue.
	
\end{proof}\subsection{Linearization along the neutral direction} \label{sec_rotation disc}We obtained a holomorphic map $\Phi: M \rightarrow T_{[z]} M, \Phi([z]) = 0 , D_{[z]} \Phi = \Id$ and 
\begin{equation}\label{1}
\Phi \circ g|_M = \Lambda \circ \Phi
\end{equation}
where $\Lambda = D_{[z]} g|_{T_{[z]} M}$. Let $\lambda=e^{ 2 \pi i \theta}, \theta \in \R \setminus \Q$ be an irrational eigenvalue of $\Lambda$ and $\C v$ be a complex line of direction $v$ in $T_{[z]}M$. The line $\C v$ is invariant by $\Lambda$, i.e. $\Lambda(\C v) = \C v$, hence $\Sigma:= \Phi^{-1}(\C v)$ is invariant by $g$. Denote by $\Gamma$ the irreducible component of $\Phi^{-1}(\C v)$ containing $[z]$.
\begin{lemma}
	Set $\Gamma_0 = \Gamma \setminus \Sing \Sigma$. Then $g(\Gamma) = \Gamma$ and $ g(\Gamma_0) = \Gamma_0$.
\end{lemma}
\begin{proof}
	On one hand, since $g$ is automorphism, it maps irreducible analytic sets to irreducible analytic sets. On the other hand, $D_{[z]} \Phi =\Id$ then by Inverse function theorem, $\Phi^{-1}(\C v)$ is smooth near $[z]$, hence $\Gamma$ is the only irreducible component of $\Phi^{-1}(\C v)$ near $[z]$. Then $g(\Gamma) = \Gamma$.
	
	Concerning $\Gamma_0$, we observe that $\Sing \Sigma = \Sigma \cap C_\Phi$ where ${C_\Phi=\{x \in M| \rank D_{x} \Phi< \dim M  \}}$ the set of critical points of $\Phi$. By differentiating $\eqref{1}$ at $x \in M$, we have
	\[D_{g(x)} \Phi \circ D_x g|_M = \Lambda \circ D_x \Phi.
	\]
	We deduce that $g(C_\Phi) = C_\Phi$. Moreover, $g(\Sigma) = \Sigma$. Thus $g( \Sing \Sigma ) = \Sing \Sigma$ and hence $g(\Gamma_0) = \Gamma_0$.
\end{proof}
Note that $\Gamma_0 \subset \Reg \Gamma$ is smooth since $\Sing \Gamma \subset \Gamma \cap \Sing \Sigma$ and in fact $\Gamma_0$ is a Riemann surface. In particular, $[z] \in \Gamma_0$.

We will prove that $\Gamma_0$ is biholomorphic to a disc and $\Phi|_{\Gamma_0}$ is conjugate to an irrational rotation. Then we can deduce from that $\Gamma = \Gamma_0$ and $\Phi_{\Gamma}$ conjugates $g|_{\Gamma}$ to an irrational rotation. Let us first recall an important theorem in the theory of dynamics in one complex dimension.
\begin{theorem}[see {\cite[Theorem 5.2]{milnor2011dynamics}}]\label{thm_hyperbolic riemann  surface}
	Let $S$ be a hyperbolic Riemann surface and let ${g: S \to S}$ be a holomorphic map with a fixed point $z$. If $z$ is an irrational fixed point with multiplier $\lambda$ then $S$ is biholomorphic to the unit disc and $g$ is conjugate to the irrational rotation $\zeta \mapsto \lambda \zeta$.
\end{theorem}
\begin{lemma}\label{lm_gamma 0}
	The Riemann surface $\Gamma_0$ is hyperbolic, $\Phi({\Gamma_0}) = \D(0,R)$ with $R \in (0,+\infty)$ and ${\Phi|_{\Gamma_0} : \Gamma_0 \to \D(0,R)}$ is a biholomorphism conjugating $g|_{\Gamma_0}$ to the irrational rotation $\zeta \mapsto \lambda \zeta$, i.e. 
	\[\Phi \circ g|_{\Gamma_0}= \lambda \cdot \Phi|_{\Gamma_0}.
	\]
\end{lemma}
\begin{proof}
	Recall that $\p(M)$ is bounded in $\C^n$. Thus $\p$ induces a non constant bounded holomorphic function from $\Gamma_0$ to $\C^n$. Therefore, $\Gamma_0$ is a hyperbolic Riemann surface. Note that $[z]$ is a fixed point of the holomorphic map $g|_{\Gamma_0}$ with the irrational multiplier $\lambda$. Then we can apply Theorem \ref{thm_hyperbolic riemann  surface} to obtain a conjugacy $\psi: \Gamma_0 \to \D(0,1)$ such that $\psi \circ g|_{\Gamma_0} = \lambda \cdot \psi.$
	
	Denote by $\Psi:= \Phi \circ \psi^{-1}: \D(0,1) \to S:= \Phi(\Gamma_0)$ (see the diagram below). Then we have $\Psi(\lambda z) = \lambda \Psi(z)$ for every $z \in \D$
	\[\xymatrix{
		\D(0,1)  \ar[r]^{\lambda \cdot} \ar@/_2pc/[dd]_{\Psi}& \D(0,1) \ar@/^2pc/[dd]^{\Psi}\\
		\Gamma_0 \ar[r]^{g} \ar[u]^{\psi} \ar[d]_{\Phi}& \Gamma_0 \ar[u]^{\psi} \ar[d]_{\Phi}\\
		S:=\Phi(\Gamma_0) \ar[r]^{\lambda \cdot}& S
	}
	\]
	It follows that $\Psi(z) = \Psi'(0) z$ for every $z \in \D$. Therefore $\Phi|_{\Gamma_0} = \Psi'(0) \cdot \psi$ is a conjugacy conjugating $g|_{\Gamma_0}$ to $z \mapsto \lambda z$. In particular, $\Phi(\Gamma_0) = \D(0,R), R =|\Psi'(0)| \in (0,+\infty)$ and $\Phi|_{\Gamma_0}$ is a biholomorphism.
\end{proof}
\begin{proposition}\label{prop_rotation disc}
	The analytic set $\Gamma$ is smooth and and the map
	\[
	\Phi|_{\Gamma}: \Gamma \to \Phi(\Gamma)=\D(0,R) \] 
	is a biholomorphic with $R \in (0,+\infty)$.
\end{proposition}
\begin{proof}
	It is enough to prove that $\Gamma = \Gamma_0$. From Lemma \ref{lm_gamma 0}, we deduce that $\Gamma_0$ is simply connected. Note that $\Gamma_0 \subset \Reg \Gamma$ is the complement of a discrete set $\Gamma \cap \Sing \Sigma$ in $\Gamma$. We denote $\hat{\Gamma}$ the normalization of $\Gamma$, a Riemann surface (see \cite{chirka2012complex}), and by $\widetilde{\Gamma}$ the universal covering of $\hat{\Gamma}$.  Since $\Gamma_0 \subset \Reg \Gamma$, the preimage $\hat \Gamma_0$ of $\Gamma_0$ by the normalization, which is isomorphic to $\Gamma_0$, is simply connected. Hence the preimage of $\hat \Gamma_0$ by the universal covering in $\widetilde{\Gamma}$ is a simply connected open subset in $\widetilde{\Gamma}$ with discrete complement. Then either $\widetilde{\Gamma}$ is biholomorphic to the unit disc (or $\C$) and $\Gamma_0 = \Gamma$ or $\widetilde{\Gamma}$ is biholomorphic to $\PP^1$ and $\Gamma \setminus \Gamma_0$ is only one point. Since $\p|_{\Gamma}$ is a non constant bounded holomorphic function valued in $\C^n$, the only case possible is that $\tilde{\Gamma}$ is biholomorphic to a disc and $\Gamma_0 = \Gamma$.
\end{proof}

Thus, we obtain a biholomorphic map $\kappa:= \left(\Phi|_{\Gamma} \right)^{-1}: \D(0,R) \rightarrow \Gamma \hookrightarrow M, \kappa(0) = [z]$ with $R \in (0,+\infty)$.

\subsection{End of the proof} Denote by $\tau = \p \circ \kappa$. Note that $\tau(0)=z$. Since $\tau(\D(0,R)) \subset \p(M)$ is bounded hence by Fatou-Riesz's theorem (see \cite[Theorem~A.3]{milnor2011dynamics}), the radial limit
\[\tau_\theta= \lim\limits_{r \rightarrow R^-} \p \circ \kappa(r e^{i \theta})
\]
exists for almost every $\theta \in [0, 2\pi)$.
\begin{remark}
	This is another advantage we mentioned in Remark \ref{Remark}.
\end{remark}
\begin{proposition}\label{radial limit}
	If $\tau_{\theta}$ exists, $\tau_{\theta} \in PC(f)$.
\end{proposition}
\begin{proof}
	Consider $\theta$ such that $\tau_{\theta}$ exists and $\tau_{\theta} \notin PC(f)$, i.e. $\tau_{\theta} \in X$. Note that 
	\[\gamma_R : [0,1] \rightarrow X
	\]
	where $\gamma_R(t)=\tau(tR e^{i \theta}), \gamma_R (1)=\tau_{\theta}$ is a well-defined path in $X$ starting at $z$ hence it defines an element in $\widetilde{X}$. Moreover, in $\widetilde{X}$, the family of paths $\{[\gamma_r]\}_{0 \le r \le R}$
	\[\gamma_r : [0,1] \rightarrow X\]
	where $\gamma_r(t)=\tau(tr e^{i \theta})$ converges to $[\gamma_R]$ as $r \rightarrow R^-$. A quick observation is that in $\widetilde{X}$, we have
	$[\gamma_r] =\kappa(re^{i \theta}) \subset M $ for every $r \in [0,1)$. Since $M$ is a closed submanifold of $\widetilde{X}$, then $[\gamma_R] \in M$ or in fact $[\gamma_R] \in \Gamma$. Recall that $\Phi: \Gamma \rightarrow \D(0,R)$ is a biholomorphic mapping hence 
	\[ re^{i \theta} = \Phi([\gamma_r]) \xrightarrow{r \rightarrow R^-} \Phi([\gamma_R]) \in \D(0,R).
	\]
	But $r e^{i \theta} \xrightarrow{r \rightarrow R^-} Re^{i \theta} \notin \D(0,R)$ which yields a contradiction. Thus $\tau_\theta \in PC(f)$.
\end{proof}
Now, denote by $Q$ a defining polynomial of $PC(f)$ then
\[Q \circ \tau: \D(0,R) \rightarrow \C
\]
has vanishing radial limit $\lim\limits_{r \rightarrow R^-} Q \circ \tau(r e^{i \theta})$ for almost every $\theta \in [0, 2 \p)$. Then, $Q \circ \tau$ vanishes identically on $\D(0,R)$ (see \cite[Theorem~A.3]{milnor2011dynamics}). In particular, $Q \circ \tau(0)=Q(z)=0$ hence $z \in PC(F)$. It is a contradiction and our proof of Proposition \ref{radial limit} and Theorem \ref{affine case} is complete.
\section{Periodic cycles in the regular locus: the transversal eigenvalue}\label{sect_in the regular locus}
Now we consider a periodic point $z$ of period $m$ in the post-critical set of a post-critically algebraic endomorphism $f$ of $\PP^n$. Note that $f^{\circ m}$ is also post-critically algebraic and $PC(f^{\circ m})$ is exactly $PC(f)$, it is enough to assume that $z$ is a fixed point. 

If $z$ is a regular point of $PC(f)$, then $T_{z} PC(f)$ is well-defined and it is a $D_{z}f$-invariant subspace of $T_{z} \PP^n$. On one hand, it is natural to expect that our method of the previous case can be extended to prove that $D_{z} f|_{T_{z} PC(f)}$ has only repelling eigenvalues (it cannot have superattracting eigenvalues, see Remark \ref{rm_non superattracting in tangent direction} below). Unfortunately, there are some difficulties due to the existence of singularities of codimension higher than $1$ that we cannot overcome easily. On the other hand, we are able to adapt our method to prove that the transversal eigenvalue with respect to $T_{z} PC(f)$, i.e. the eigenvalue of $\overline{D_{z} f}: T_{z} \PP^n / T_{z} PC(f) \rightarrow T_{z} \PP^n / T_{z} PC(f)$, is repelling. More precisely, we will prove the following result.
\begin{proposition}\label{prop_transversal eigenvalue}
	Let $f$ be a post-critically algebraic endomorphism of $\PP^n$ of degree ${d \ge 2}$ and $z \in \Reg PC(f)$ be a fixed point. Then the eigenvalue of the linear map $\overline{D_{z} f}: T_{z} \PP^n / T_{z} PC(f) \rightarrow T_{z} \PP^n/ T_{z} PC(f)$ is either repelling or superattracting.
\end{proposition} 
By Remark \ref{Remark}, it is equivalent to prove the following result.
\begin{proposition}\label{prop_transversal eigenvalue affine}
	Let $f$ be a post-critically algebraic non-degenerate homogeneous polynomial endomorphism of $\C^n$ of degree ${d \ge 2}$  and $z \in \Reg PC(f)$ be a fixed point. Then the eigenvalue of the linear map $\overline{D_{z} f}: T_{z} \C^n / T_{z} PC(f) \rightarrow T_{z} \C^n / T_{z} PC(f)$ is either repelling or superattracting.
\end{proposition}
Since $PC(f)$ has codimension one, $\overline{D_{z}f}$ has exactly one eigenvalue and we denote it by $\lambda$. The value $\lambda$ is also an eigenvalue of $D_{z} f$. The proof of Proposition \ref{prop_transversal eigenvalue affine} will occupy the rest of this section.
\subsection{Strategy of the proof}Denote by $X= \C^n \setminus PC(f)$. 
\begin{itemize}
	\item[Step 1.] We first prove that if $\lambda \neq 0$ then $|\lambda|\ge 1$ and $z$ is not a critical point. Then we prove that $|\lambda|=1$ will lead to a contradiction. By assuming $|\lambda|=1$, following from the discussion in Section \ref{sect_non critical fixed point}, there exists an eigenvector $v$ of $D_z f$ corresponding to $\lambda$ such that $v \notin T_{z} PC(f)$.
	
	Our goal is to build a holomorphic map $\tau: \D(0,R)^* \to X$ such that $\tau$ can be extended holomorphically to $\D(0,R)$ so that $
	\tau(0) = z, \tau'(0) = v$. Then, we show $\tau$ has radial limit almost everywhere and these radial limits land on $PC(f)$ whenever they exist. The construction of $\tau$ occupies Step 2 to Step 6 and the contradiction will be deduced in Step 7.
	\item[Step 2.]  We construct a connected complex manifold $\widetilde{X}$ of dimension $n$ with two holomorphic maps $\p : \widetilde{X} \rightarrow X, g: \widetilde{X} \rightarrow \widetilde{X}$ such that 
	\[f \circ \p \circ g = \p.
	\]
	\item[Step 3.] We prove that $\{g^{\circ j} \}_j$ is a normal family. Then we extract a subsequence $\{g^{\circ j_k}\}_k$ converging to a retraction $\rho: \widetilde{X} \rightarrow \widetilde{X}$, i.e. $\rho \circ \rho = \rho.$
	\item[Step 4.] We will study $M = \rho(\widetilde{X})$. More precisely, we will prove that ${\pi}(M)$ can be extended to a center manifold of $f$ at $z$. 
	\item[Step 5.] We construct a holomorphic map $\Phi: M \rightarrow E_n$ which semi-conjugates $g$ to the restriction of $(D_{z} f)^{-1}$ to the neutral eigenspace ${E_n}$ (see Proposition \ref{prop_properties of affine homogeneous maps}.c).
	\item[Step 6.] We prove that there exists an irreducible component $\Gamma$ of $\Phi^{-1}(\C v)$ which is smooth and biholomorphic to the punctured disc. More precisely, we prove that $\Phi(\Gamma) = \D(0,R)^*$ with $R \in (0,+\infty)$ and the map $\tau:= \p \circ \left( \Phi|_{\Gamma}\right)^{-1}$ extends to a holomorphic map from $\D(0,R)$ to $\C^n$ so that ${\tau(0) =z, \tau'(0) = v}$.
	\item[Step 7.]  We prove that the map $\tau$ has radial limit almost everywhere and the limit belong to $PC(f)$ if it exists. It implies that $\p \circ \tau \subset PC(f)$ which yields a contradiction to the fact that $v \notin T_{z} PC(f)$. This means that the assumption $|\lambda|=1$ is false thus Proposition \ref{prop_transversal eigenvalue affine} is proved. 
\end{itemize}
\subsection{Existence of the transversal eigenvector}
Let us recall the following result due to Grauert.
\begin{proposition}[\cite{grauert1958komplexe}, Satz 10]
	Let $U,V$ be open neighborhood of $0$ in $\C^n$ and let $f: U \to V$ be a holomorphic branched covering of order $k$ ramifying over $V_f=\{\zeta_n=0\} \cap V$. Then there exists a biholomorphism $\Phi: U \to W$ such that the following diagram commutes
	\[
	\xymatrix{
		& W \ar[d]^{(\zeta_1,\ldots,\zeta_n) \mapsto (\zeta_1,\ldots,\zeta_{n-1},\zeta_n^k)}\\
		U \ar[r]_{f} \ar[ur]^{\Phi} & V }\]
	In particular, the branched locus $B_f = \Phi^{-1}(\{\zeta_n=0\} \cap W)$ is smooth and $f|_{B_f}: B_f \to V_f$ is a biholomorphism.
\end{proposition}
This is in fact a local statement and we can apply it to a post-critically algebraic non-degenerate homogeneous polynomial endomorphism of $\C^n$  to obtain the following result.
\begin{proposition}[\cite{ueda1998critical}, Lemma 3.5]\label{prop_locally biholomorphism in tangent direction}
	Let $f$ be a post-critically algebraic non-degenerate homogeneous polynomial endomorphism of $\C^n$ of degree $d \ge 2$. Then \[f^{-1}( \Reg PC(f)) \subset \Reg PC(f)\] and \[
	f: f^{-1}(\Reg PC(f)) \to \Reg PC(f)\]
	is locally a biholomorphism.
\end{proposition}
\begin{remark}\label{rm_non superattracting in tangent direction}
	In particular, Proposition \ref{prop_locally biholomorphism in tangent direction} implies that if $f$ has a fixed point $z \in \Reg PC(f)$ then $D_{z} f|_{T_{z} PC(f)}$ is invertible. Hence $D_{z} f|_{T_{z} PC(f)}$ does not have any superattracting eigenvalue.
\end{remark}
If $\lambda \neq 0$ then $z$ is not a critical point. By Proposition \ref{prop_properties of affine homogeneous maps}, the modulus of $\lambda$ is at least $1$. Then we will prove Proposition \ref{prop_transversal eigenvalue} by contradiction by assuming that $|\lambda|=1$. If $|\lambda|=1$, there exists an associated eigenvector $v$ of $D_{z} f$ such that $v \notin T_{z} PC(f)$. Indeed, note that 
\[\ssp(D_z f) = \ssp(D_z f|_{T_{z} PC(f)}) \cup \ssp(\overline{D_z f})
\]
where $\ssp(\overline{D_z f})$ has only one eigenvalue $\lambda$ of modulus one. Then the repelling eigenspace $E_r$ is included in $T_z PC(f)$. The diagonalizability of $D_z f|_{E_n}$ implies that $E_n$ is generated by a basis of eigenvectors. The vector $v$ is such an eigenvector which is not in $T_z PC(f)$.

\subsection{$(X,z)$-homotopy and related constructions}Denote by $X = \C^n \setminus PC(f)$.
\subsubsection{Construction of $\widetilde{X}$} We construct a complex manifold $\widetilde{X}$, a covering map $\p: \widetilde{X} \rightarrow X$ and a holomorphic map $g: \widetilde{X} \rightarrow \widetilde{X}$ such that 
\[f \circ \p \circ g = \p.
\] Denote by 
\[\Xi=\{\gamma : [0,1]\rightarrow \C^n \text{ continuous map such that } \gamma(0)=z, \gamma((0,1]) \subset X \}
\]
the space of paths starting at $z$ and varying in $X$. Let $\gamma_0,\gamma_1 \in \Xi$. We say that $\gamma_0$ and $\gamma_1$ are \textit{$(X,z)$-homotopic} if 
there exists a continuous map $H: [0,1] \times [0,1] \rightarrow \C^n $ such that 
\[H(0,s)=z,H(1,s)=\gamma_0(1)=\gamma_1(1),
\]
\[H(t,0)=\gamma_0(t),H(t,1)=\gamma_1(t),
\]
\[H(t,s) \subset X \, \forall t \neq 0.
\]

\begin{figure}[h]
	\centering{
		\fbox{\resizebox{100mm}{!}{\includegraphics[scale=1]{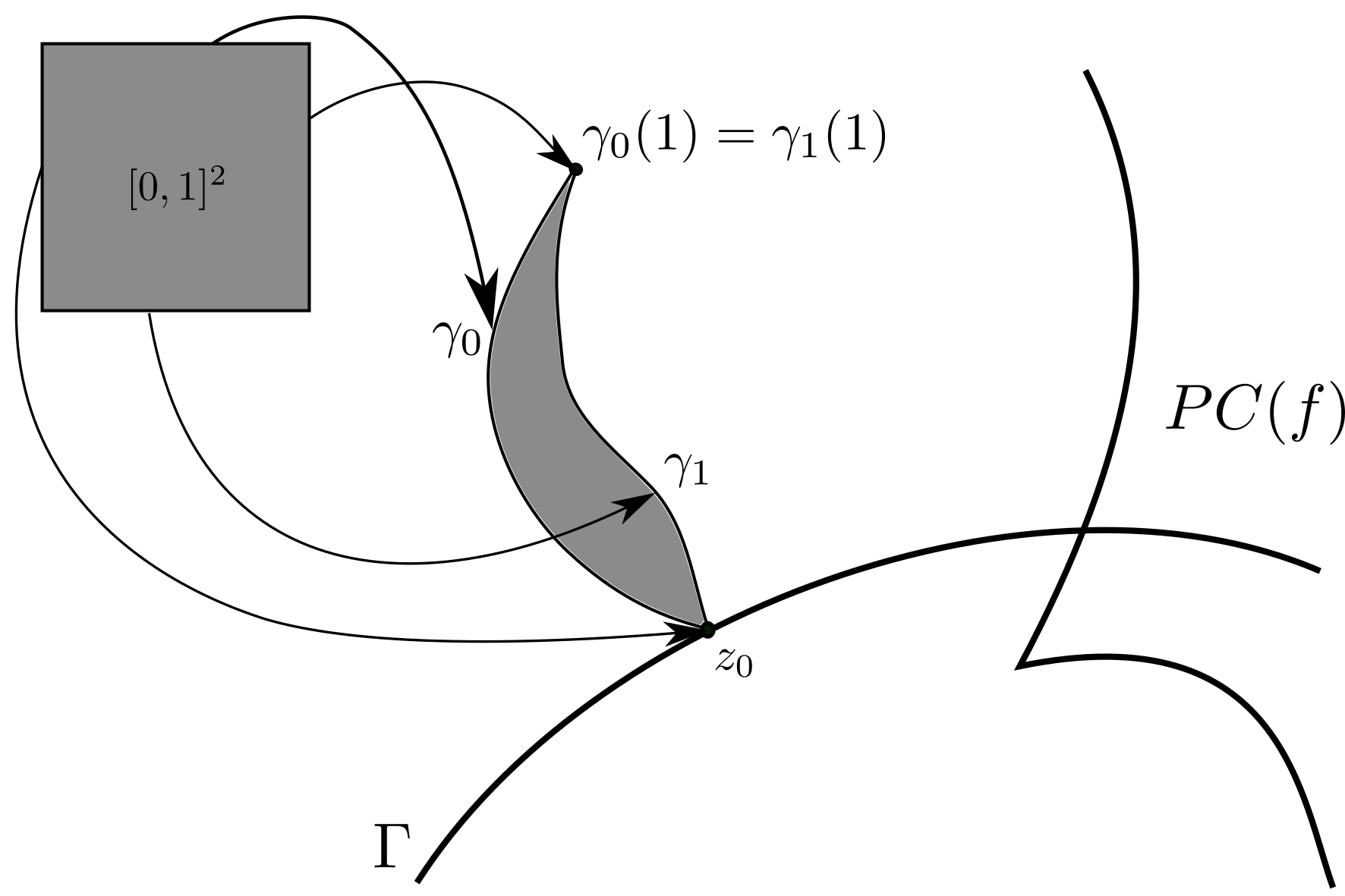}}}
		\caption{Two paths $\gamma_0,\gamma_1$ which are $(X,z)$-homotopic}
		\label{fig:relative}
	}
\end{figure}
Denote by $\gamma_0 \sim_X \gamma_1$ if $\gamma_0$ and $\gamma_1$ are $(X,z)$-homotopic. In other words, $\gamma_0$ and $\gamma_1$ are homotopic by a homotopy of paths $\{\gamma_t,t \in [0,1]\}$ such that $\gamma_t \in \Xi$ for every $t$. It is easy to see that $(X,z)$-homotopy is an equivalence relation on $\Xi$. Denote by $\widetilde{X}$ the quotient space of $\Xi$ by this relation and by $[\gamma]$ the equivalent class of $\gamma \in \Xi$. Denote by 
\[\p: \widetilde{X} \rightarrow X, \p([\gamma]) = \gamma(1)
\]
the projection. We endow $\widetilde{X}$ with a topology constructed in the same way as the topology of a universal covering. More precisely, let $\mathfrak{B}$ be the collection of simply connected open subsets of $X$. Note that $\mathfrak{B}$ is a basis for the usual topology of $X$. We consider the topology on $\widetilde{X}$ which is defined by a basis of open subsets $\{U_{[\gamma]}\}_{U \in \mathfrak{B}, [\gamma] \in X}$ where $\gamma(1) \in U$ and 
\[U_{[\gamma]}= \{[\gamma * \alpha]| \alpha \, \mbox{is a path in $U$ starting at $\gamma(1)$}\}.
\]
We can transport the complex structure of $X$ to $\widetilde{X}$ and this will make $\widetilde{X}$ a complex manifold of dimension $n$. Note that $\p$ is also a holomorphic covering map. 
\subsubsection{Lifts of inverse branches of $f$} We will construct a holomorphic mapping \[g: \widetilde{X} \rightarrow \widetilde{X}\] which is induced by the pullback action of $f$ on paths in $\Xi$.
\begin{lemma}\label{pulling back}
	Let $\gamma$ be a path in $\Xi$. Then there exists a unique path $f^* \gamma \in \Xi$ such that $f \circ f^* \gamma = \gamma$.
\end{lemma}

\begin{proof}Since $f$ is locally invertible at $z$ and since $f^{-1}(X) \subset X$, there exists $t_0 \in [0,1]$ such that $\gamma|_{[0,t_0]} \in \Xi$ and $f^{-1}\circ \gamma|_{[0,t_0]}$ is a well-defined element in $\Xi$. Then the path $f^* \gamma$ is the concatenation of $f^{-1} \circ \gamma|_{[0,t_0]}$ with the lifting $f^*\gamma|_{[t_0,1]}$ of the path $\gamma|_{[t_0,1]}$ by the covering $f: f^{-1}(X) \to X$. This construction does not depend on the choice of $t_0$. 
	\begin{figure}[h]\label{fig:pullback def}
		\centering{
			\fbox{\resizebox{100mm}{!}{\includegraphics[scale=0.8]{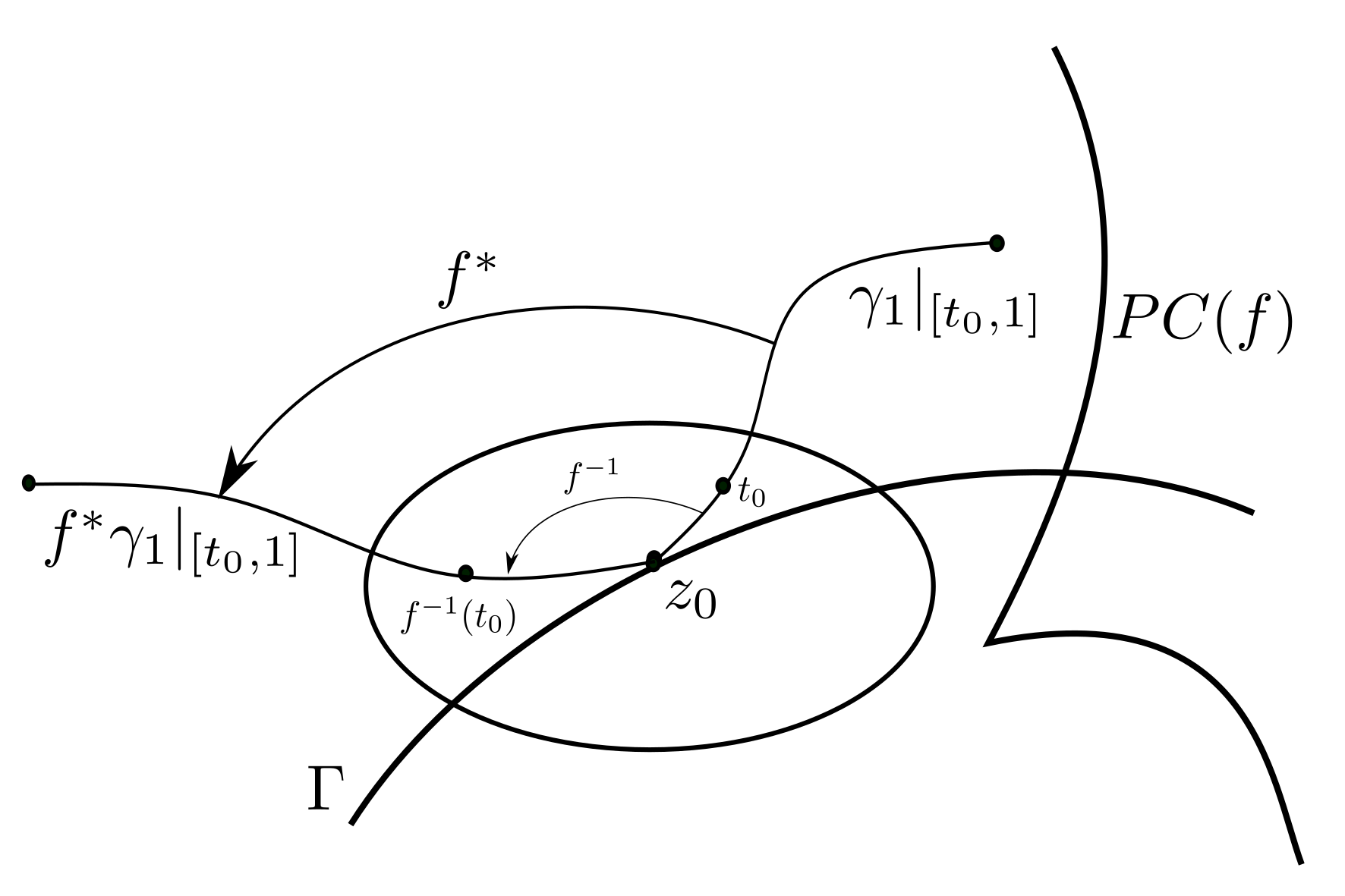}}}
			\caption{Pulling back of an element in $\Xi$.}
			
		}
	\end{figure}
\end{proof}

\begin{lemma}
	Let $\gamma_0,\gamma_1 \in \Xi$. If $[\gamma_0]=[\gamma_1]$ then $[f^*\gamma_0] = [f^*\gamma_1]$.
\end{lemma}

\begin{proof}
	Lemma \ref{pulling back} implies that the pull-back of a homotopy of paths in $\Xi$ between $\gamma_0$ and $\gamma_1$ is a homotopy of paths in $\Xi$ between $f^*\gamma_0$ and $f^*\gamma_1$. See also Figure \ref{fig_pullback}.
\end{proof}
\begin{figure}[H]\label{fig_pullback}
	\centering{
		\fbox{\resizebox{100mm}{!}{\includegraphics[scale=1]{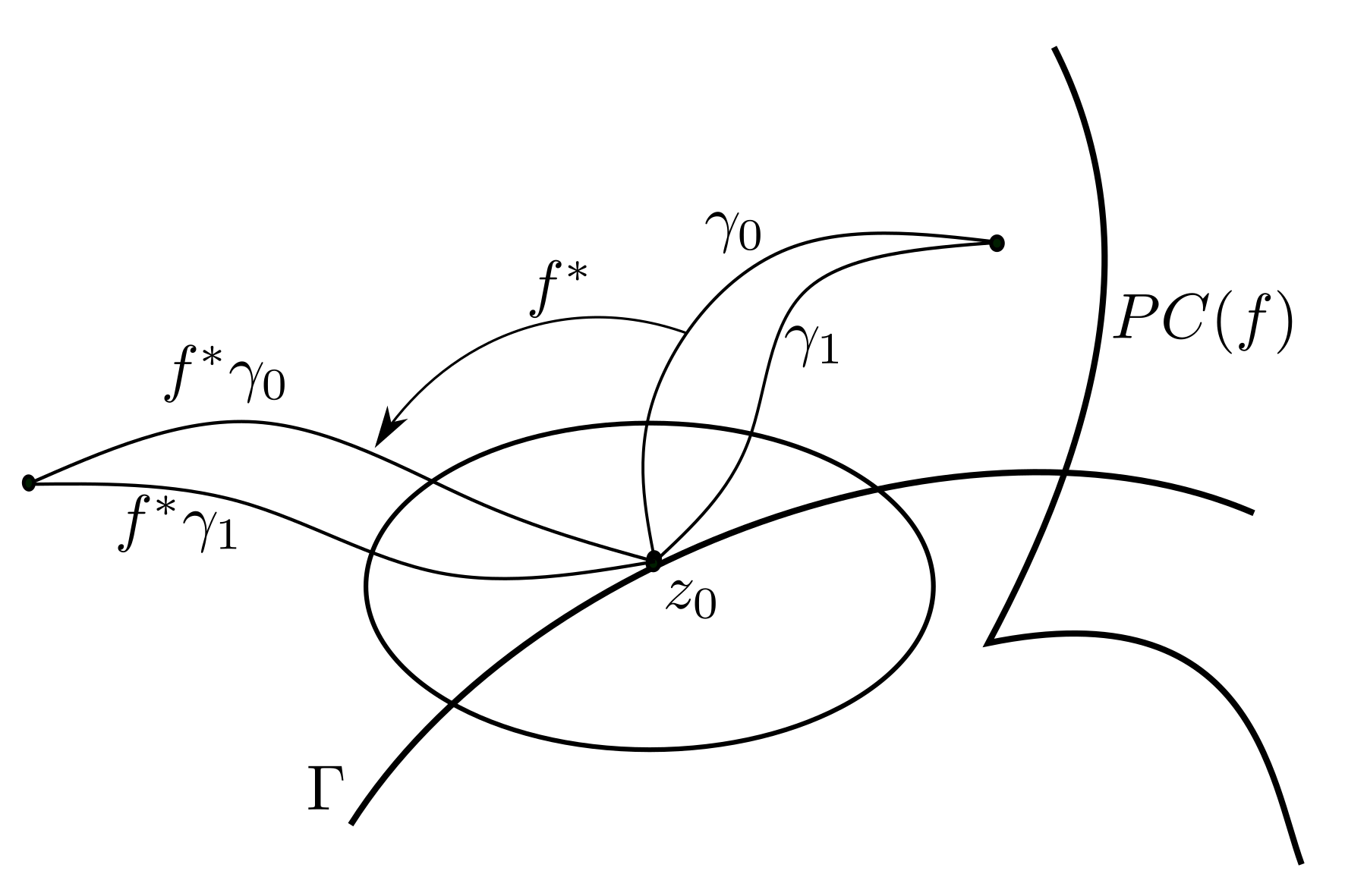}}}
		\caption{Pull-back preserves the $(X,z)$-homotopic paths}
		
	}
\end{figure}
The two previous lemmas allow us to define a map $g: \widetilde{X} \rightarrow \widetilde{X}$ as follow
\[g([\gamma]) = [f^*\gamma].
\]
Then $f \circ \p \circ g = \p.$
\subsubsection{The connectedness of $\widetilde{X}$.}
The connectedness of $\widetilde{X}$ is not obvious from the construction. We will introduce a notion of \textit{regular neighborhood} which is not only useful in proving $\widetilde{X}$ is connected but also very important later.
\begin{definition}\label{def_regular neighborhood}
	A bounded open subset $W$ of $\C^n$ containing $z$ is called a \textit{regular neighborhood }of $z$ if:
	\begin{itemize}
		\item[1.] $(W,W \cap PC(f))$ is homeomorphic to a cone over $(\partial W, \partial W \cap PC(f))$ with a vertex at $z$.
		\item[2.] For every path $\gamma_0,\gamma_1 \in \Xi$ such that $\gamma_0([0,1]),\gamma_1([0,1]) \subset W$ and $\gamma_0(1) = \gamma_1(1)$. Then $\gamma_0 \sim_X \gamma_1$.
	\end{itemize}
\end{definition}

Let $W$ be an open subset of $\C^n$ containing $z$. Set 
\[\widetilde{W}= \{[\gamma]| \gamma \in \Xi, \gamma((0,1]) \subset W \}.
\]
\begin{lemma}\label{neighbor path connected}
	If $W$ is a regular neighborhood of $z$, then $\p: \widetilde{W} \to W \setminus PC(f)$ is a biholomorphism.
\end{lemma}
\begin{proof} 
	We can observe that $\widetilde{W}$ is open. Indeed, for an element $[\gamma]$ in $\w$, let $U$ be an open set in $W \setminus PC(f)$ containing $\gamma(1)$, then
	$U_{[\gamma]} \subset \w$.
	The projection $\p|_{\w}: \w \rightarrow  W \setminus PC(f)$ is surjective since $W \setminus PC(f)$ is path-connected. So we have to prove that $\p: \w \rightarrow W \setminus PC(f)$ is injective. Indeed, let $z$ be a point in $W \setminus PC(f)$ and let $[\gamma_0],[\gamma_1]$ be two elements in $\w$ such that $\p([\gamma_0])=\p([\gamma_1])=z$, i.e. $\gamma_0(1)=\gamma_1(1)$. Since $W$ is regular, we have $\gamma_0 \sim_X \gamma_1$ or $[\gamma_0]=[\gamma_1]$. So $\p|_{\w}$ is injective hence biholomorphic.
\end{proof}
In particular, $\widetilde{W}$ is path-connected. If $\gamma:[0,1] \to X$ is a path in $\Xi$ then the path $\gamma_s:[0,1] \to \C^n$ defined by $\gamma_s(t)= \gamma(t(1-s))$ also belongs to $\Xi$ and $\gamma_s([0,1]) = \gamma([0,1-s])$. It follows that every element in $\widetilde{X}$ can be joined by paths to an element in $\widetilde{W}$ thus $\widetilde{X}$ is path-connected hence connected. 

Now we will prove that we can indeed find a regular neighborhood when $z$ is a regular point of $PC(f)$. Let $(\zeta_1,\ldots,\zeta_n)$ be a local coordinates vanishing at $z$ in which $PC(f)$ is given by $\{\zeta_1=0\}$. Let $U$ be the unit polydisk centered at $z$.  
\begin{proposition}\label{prop_existence of regular neighborhood}
	Any polydisc centered at $z$ in $U$ is a regular neighborhood.
\end{proposition}
\begin{proof}
	Let $\gamma_0$ and $\gamma_1$ be two elements of $\Xi$ such that $\gamma_0([0,1]), \gamma_1([0,1])\in U$ and $\gamma_0(1) = \gamma_1(1)$. Consider the loop $\eta = \gamma_0 *(-\gamma_1)$ and the continuous map $H:[0,1]\times [0,1]\to \C^n$ defined by 
	\[H(t,s) = s\cdot \eta(t).\]
	The loop $\eta$ bounds $H([0,1]\times [0,1] \setminus \{(0,0)\})\subset X$, which implies that $\gamma_0\sim_X \gamma_1$. 
\end{proof}

\begin{remark}\label{rm_regular neighborhood} The construction above also implies that $z$ admits a basis of neighborhoods consisting of regular neighborhoods. 
\end{remark}
\subsubsection{Dynamics of $g$ on regular neighborhoods.}\label{sec_regular neighborhoods}
Let $W$ be a regular neighborhood of $z$ and \[\sigma: W \setminus PC(f) \rightarrow \widetilde{W} \] be the inverse of $\p|_{\w}$. Note that if $W$ is constructed as above, then by Proposition \ref{prop_properties of affine homogeneous maps}.b, there exists a family of holomorphic maps $h_j: W \rightarrow \C^n , j \ge 1$  such that
\[h_j(z) = z, f^{\circ j } \circ h_j = \Id_W.
\]  
We can deduce from the definition of $g$ that for every $j \ge 1$, $\p \circ g^{\circ j} \circ \sigma = h_j|_{W \setminus PC(f)}$. More precisely, let $[\gamma] \in \widetilde{W}$, i.e. $\gamma((0,1]) \subset W \setminus PC(f)$, then by definition, we have for every $j \ge 1$,
\begin{equation}\label{local representation of g}
g^{\circ j}([\gamma]) = [h_j \circ \gamma].
\end{equation}

Recall that the family $\{h_j: W \to \C^n\}_j$ is normal. The assumption $|\lambda|=1$ allows us to control the value taken by any limit maps of this family. Note that $f^{-1}(X) \subset X$ hence $h_j ( W \setminus PC(f) ) \subset X$ for every $j \ge 1$.
\begin{lemma}\label{lm_limit of inverse branches}
	Let $h=\lim\limits_{s \to \infty} h_{j_s}$ be a limit map of $\{h_j: W \to \C^n\}_j$. Then \[h(W \setminus PC(f)) \subset X.\]
\end{lemma}
\begin{proof}
	Recall that $PC(f)$ is the zero locus of a polynomial $Q: \C^n \to \C$. Then \[Q \circ h= \lim\limits_{s \to +\infty} Q \circ h_{j_s}.\] Since $h_{j_s}(W \setminus PC(f)) \subset X$, the map $Q \circ h|_{W \setminus PC(f)}$ is the limit of a sequence of nonvanishing holomorphic functions. By Hurwitz's theorem, $Q \circ h|_{W \setminus PC(f)}$ is either a nonvanishing function or identically $0$, i.e. either $h(W \setminus PC(f)) \subset X$ or $h(W \setminus PC(f)) \subset PC(f)$. 
	
	Let $v$ be an eigenvector of $D_{z} f$ associated to the eigenvalue $\lambda$. Then we have
	$D_{z} h_j(v) = \frac{1}{\lambda^j}v$. Hence 
	\[D_{z} h(v) = \frac{1}{\lambda'} v
	\]
	for some limit value $\lambda'$ of $\{\lambda^j\}_j$. Since we assumed that $|\lambda|=1$ hence $|\lambda'|=1$. The fact that $v \notin T_{z} PC(f)$ implies that $D_{z} h(v) \notin T_{z} PC(f)$. Consequently, we have \[h(W \setminus PC(f)) \cap X \neq \emptyset\] thus $h(W \setminus PC(f)) \subset X$.
\end{proof}
\subsection{Normality of family of maps lifted via the relative homotopy}  We will prove that $\{g^{\circ j}: \widetilde{X} \to \widetilde{X} \}_j$ is normal. Following \ref{sec_step 2 case 1}, it is enough to prove two following lemmas.
\begin{lemma}\label{lm_lifitng case 2}
	The family $\{k_j= \p \circ g^{\circ j}: \widetilde{X} \to X\}_j$ is normal and any limit map can be lifted by $\p$ to a holomorphic endomorphism of $\widetilde{X}$.
\end{lemma}
\begin{proof}
	Note that $\{k_j: \widetilde{X} \to \C^n \}_j$ is locally uniformly bounded hence normal (see Proposition \ref{backward manifold}). Consider a limit map $k$ of this family, by using Hurwitz's theorem, we deduce that either $k(\widetilde{X}) \subset X$ or $k(\widetilde{X}) \subset PC(f)$. 
	
	Let $W$ be a regular neighborhood of $z$ then there exists a family $\{h_j: W \to \C^n\}_j$ of $f^{\circ j}$ fixing $z$ (see \ref{sec_regular neighborhoods}). We have \[k_j|_{\widetilde{W}} \circ \sigma = h_j|_{W \cap X}\]  where $\sigma: W \setminus PC(f) \to \widetilde{W}$ is the section of $\p|_{\widetilde{W}}$. Therefore $k|_{\widetilde{W}} \circ \sigma$ is a limit map of $\{h_j|_{W \setminus PC(f)}\}$. Lemma \ref{lm_limit of inverse branches} implies that $k(\widetilde{W}) \subset X$, and hence, $k(\widetilde{X}) \subset X$. Thus $\{k_j: \widetilde{X} \to X \}_j$ is normal and any limit map takes values in $X$.  
	
	We shall now show that the map $k \colon \widetilde{X} \to X$ can be lifted to a map from $\widetilde{X}$ to $\widetilde{X}$. Set 
	\[
	h \colon= k|_{\widetilde{W}} \circ \sigma
	\]
	then $h$ is a limit map of $\{ h_j|_{W \setminus PC(f)} \}$. For each element $[\gamma] \in\widetilde{X}$, we denote by $\eta$ the image of $\gamma$ under the analytic continuation of $h$ along $\gamma$. Note that $h$ fixes $z$. Thus, Lemma \ref{lm_limit of inverse branches} implies that $\eta \in \Xi$. This construction does not depends on the choice of $\gamma$ in the equivalence class $[\gamma]$. Thus, the map \[
	\begin{array}{cccc}
	\widetilde{k} \colon & \widetilde{ X} & \to & \widetilde{X}\\
	& [\gamma] & \mapsto & [\eta]
	\end{array}
	\]
	is well-defined. The map $\pi \circ \widetilde{k}$ coincides with $k$ on an open set $\widetilde{W}$ in $\widetilde{X}$, and hence coincides with $k$ on $\widetilde{X}$. In other words, $\widetilde{k}$ is a lifted map of $k$ by $\pi$.
\end{proof}
Hence we deduce that:
\begin{proposition}
	The family $\{g^{\circ j}: \widetilde{X} \to \widetilde{X} \}_j$ is normal.
\end{proposition}
\begin{proof}
	Let $\{g^{\circ j_s}\}_s$ be sequence of iterates of $g$. Extracting subsequences if necessary, we can assume that $\{k_{j_s}\}_s$ converges locally uniformly to a holomorphic map $k: \widetilde{X} \to X$. By Lemma \ref{lm_lifitng case 2}, there exists a holomorphic map $\widetilde{k} : \widetilde{ X } \to\widetilde{X}$ so that $\p \circ \widetilde{k} = k$. We will prove that $\{g^{\circ j_s}\}_s$ converges locally uniformly to $\widetilde{k}$. Applying \cite[Theorem 4]{andreian2003coverings}, it is enough to prove that there exists an element $[\gamma] \in \widetilde{X}$ such that $g^{\circ j_s}([\gamma])$ converges to $\widetilde{k}([\gamma])$.
	
	We consider a regular neighborhood $W$ of $z$ and the family $\{h_j:W \to \C^n\}_j$ of $f^{\circ j}$ fixing $z$ (see \ref{sec_regular neighborhoods}). Let $[\gamma]$ be an element in $\widetilde{W}$\footnote{In Section \ref{sec_step 2 case 1}, we choose an element representing $z$. Such an element does not exist in this case but it is enough to consider an element representing a point near $z$} and an associated path $\widetilde{\gamma}: [0,1] \to \widetilde{X}, \widetilde{\gamma}(t) = [\gamma|_{[0,t]}]$ in $\widetilde{X}$. Since $[\gamma] \in  \widetilde{W}$ , $\widetilde{\gamma} = \sigma \circ \gamma$. Then 
	\[\begin{array}{ccl}
	\widetilde{k}([\gamma]) = [k \circ \widetilde{\gamma}] &=& \lim\limits_{s \to \infty} [k_{j_s} \circ \widetilde{\gamma}]\\
	&=& \lim\limits_{s \to \infty} [h_{j_s} \circ \gamma] =\lim\limits_{s \to \infty} g^{\circ j_s} ([\gamma]) .
	\end{array}
	\]
	Thus we conclude the proof the proposition.
\end{proof}
Following \cite[Corollary 2.1.29]{abate1989iteration}, the normality of $\{g^{\circ j}\}_j$ implies that
\begin{itemize}
	\item there exists a subsequence $\{g^{\circ j_k}\}_k$ converging to a holomorphic retraction \[\rho: \widetilde{X} \to \widetilde{X} \]of $\widetilde{X}$, i.e. $\rho \circ \rho =\rho.$,
	\item by \cite{cartan1986retractions}, the image ${M}=\rho(\widetilde{X})$ is a closed submanifold of $\widetilde{X}$,
	\item by \cite[Corollary 2.1.31]{abate1989iteration}, ${M}$ is invariant by $g$ and $g|_{{M}}$ is an automorphism. 
\end{itemize}   

\subsection{Existence of the center manifold} We will study the dynamics of $g$ restricted on $M$. The difference between the construction of universal covering used in the first case (the fixed point is outside $PC(f)$) and the construction of $\widetilde{X}$ in this case is that $\widetilde{X}$ does not contain a point representing $z$. Hence it is not straight forward that we can relate the dynamics of $g$ on $M$ with the dynamics of $f$ near $z$. 

We consider the objects introduced in Section \ref{sec_regular neighborhoods}. In particular, we consider a regular neighborhood $W$ of $z$ in $\C^n$ and the family $\{h_j: W \to \C^n\}_j$ of inverse branches fixing $z$ of $f^{\circ j}$ on $W$. Recall that 
\[
\sigma \colon W \setminus PC(f) \to \tilde{W}
\]
is the inverse of the biholomorphism $\pi: \tilde{W} \to W \setminus PC(f)$ and that $\lim\limits_{k \to \infty} g^{\circ j_k}= \rho$ is a holomorphic retraction on $\widetilde{X}$. 

Define a holomorphic map $\widetilde{H} : W \setminus PC(f) \to \C^n$ as follows:
\[\widetilde{H} = \p \circ \rho \circ \sigma = \lim\limits_{k \to +\infty} \p \circ g^{\circ j_k} \circ \sigma.
\]
By \eqref{local representation of g}, we have $\widetilde{H}= \lim\limits_{k \rightarrow +\infty} h_{j_k}|_{W \setminus PC(f)}$. Since $\{h_j: W \to \C^n\}_j$ is normal, by passing to subsequences, we can extend $\widetilde{H}$ to a holomorphic map $H: W \to \C^n$ such that $H = \lim\limits_{ k\to +\infty} h_{j_k} \colon W \to \C^n$.

Note that $h_j(z) = z$ for every $j \ge 1$. Then $H(z) = z$. By continuity of $H$, there exists an open neighborhood $U$ of $z$ in $W$ such that $H(U) \subset W$. Note that we choose $U$ to be a regular neighborhood of $z$ (see Remark \ref{rm_regular neighborhood}) and we can shrink $U$ whenever we need to. Recall that for every $[\gamma] \in \widetilde{W}$, we have $g^{\circ j} ([\gamma]) = [h_j \circ \gamma]$. Then for ${[\gamma] \in \widetilde{U} := \sigma(U \setminus PC(f))}$, we have
\[\rho([\gamma]) = \lim\limits_{k\to \infty} g^{\circ j_k}([\gamma]) = \lim\limits_{k \to \infty} [h_{j_k} \circ \gamma]= [H \circ \gamma] \subset \w
\]
In other words, $\rho(\widetilde{U}) \subset \w$. Hence 
\begin{equation}\label{eq_local retraction of z}
{H}(U \setminus PC(f)) = \p \circ \rho \circ \sigma(U \setminus PC(f)) \subset W \setminus PC(f).
\end{equation}
Moreover,  since $\sigma \circ \p|_{\w} = \Id_{\w}$, the composition  \[\widetilde{H} \circ \widetilde{H} = \p \circ \rho \circ \sigma \circ \p \circ \rho \circ \sigma.
\] is well defined on $U \setminus PC(f)$ and equals to $\widetilde{H}|_{U \setminus PC(f)}$. Since $H$ is the extension of $\widetilde{H}$, we deduce that
\[H \circ H (U) = H(U).
\]
\begin{proposition}
	The set $H(U)$ is a submanifold of $W$ containing $z$ whose dimension is the number of neutral eigenvalues of $D_{z} f$ counted with multiplicities. Moreover, $T_{z} H(U) =E_n$ and $D_{z}f|_{T_{z} H(U)}$ is diagonalizable.
\end{proposition}
\begin{proof}
	The first assertion is due to \cite{cartan1986retractions} since $H \circ H = H$ on $U$. The rest are consequences of the fact that $H$ is a limit map of the family $\{h_j: W \to \C^n\}_j$ of inverse branches fixing $z$ of $f$ (see Corollary \ref{cor_non attracting}).
\end{proof}
\begin{lemma}
	$\dim {M} = \dim H(U)$.
\end{lemma}
\begin{proof}
	Since $\widetilde{X}$ is connected, ${M}$ is also a connected complex manifold. Thus $\dim {M} = \rank D_x \rho$ for every $x \in \widetilde{M}$. In particular, if we choose $x =\sigma (z)$ with $w \in U \setminus PC(f)$ then \[\rank D_x \rho =\rank_w H = \dim H(U)\]
	thus $\dim {M} = \dim H(U)$. 
\end{proof}

In other words, \[ M_X : = \p(M) \cup H(U)\] is a submanifold of $\C^n$ in a neighborhood of $z$. Moreover, following Proposition \ref{backward manifold}, we can deduce that $\p(M)$ is a bounded set in $\C^n$. 
\subsection{Semiconjugacy on the center manifold} Denote by $\Lambda= (D_z f|_{E_n})^{-1}$. We will construct a holomorphic $\Phi: M \to E_n$ such that \[\Phi \circ g|_M = \Lambda \circ \Phi.\] The construction follows the idea in Section \ref{sec_semiconjugacy} and the connection established in {Step 3} between $g$ and inverse branches of $f$ at $z$.

\begin{lemma}
	There exists a holomorphic map $\Phi: M \to E_n$ such that $\Phi \circ g = \Lambda \circ \Phi$. 
\end{lemma}
\begin{proof} We consider a holomorphic map $\varphi: M \to E_n$ constructed as the following composition
	\[M \hookrightarrow \widetilde{X} \xrightarrow{\p} \C^n \xrightarrow{\delta} T_z \C^n \xrightarrow{\p_{E_r}} E_n
	\]
	where $\delta: \C^n \to T_z \C^n$ is a holomorphic map such that $\delta(z) = 0, D_z \delta = \Id$, $\p_{E_r}: T_z \C^n \to E_n$  is a projection on $E_n$ parallel to $E_r$. Note that since $\p(M)$ is bounded, $\varphi(M)$ is also a bounded set in $E_n$.
	
	
	We consider the family \[\Lambda^{-j} \circ \varphi \circ g^{\circ j} : M \to E_n, j \ge 0.\]Since $\Lambda$ is diagonalizable with only neutral eigenvalues, this family is uniformly bounded hence so is the family of its Ces\`aro averages $\{\Phi_N= \frac{1}{N} \sum_{j=0}^{N-1} \Lambda^{-j} \circ \varphi \circ g^{\circ j} \}_N$. Therefore, $\{\Phi_N\}_N$ is normal and every limit map $\Phi$ of $\{\Phi_N\}_N$ satisfies that
	\[\Phi \circ g = \Lambda \circ g.
	\]
	Note that $\Phi(M)$ is also a bounded set in $E_n$.
\end{proof}
Let us fix such a limit map $\Phi =\lim\limits_{k \to \infty} \Phi_{N_k}$. We will prove that $\Phi$ restricted to $\widetilde{U}  \cap M$ is a biholomorphism. In order to do so, we consider the following holomorphic function \[{\Phi_1:= \Phi \circ \sigma|_{H(U)}: H(U \setminus PC(f)) \to \C^n}.\]
Since $H(U \setminus PC(f)) \subset \pi(M)$ is a bounded set in $\C^n$, the map $\Phi_1$ is bounded hence we can extend it to a holomorphic function on $H(U)$. By an abuse of notations, we denote the extension by $\Phi_1$. We will prove that $\Phi_1$ is invertible in a neighborhood of $z$ in $H(U)$.

More precisely, on $H(U) \setminus PC(f)$, we have
\[ 	\begin{array}{ccl}
\Phi_1&=& \lim\limits_{k \to \infty} \frac{1}{N_k} \sum_{j =0}^{N_k -1} \Lambda^{-j} \circ \varphi \circ g^{\circ j} \circ \sigma|_{h(U) \setminus PC(f)}\\
\end{array}\]
Consider the map $\varphi_1:  \p(M) \cup H(U) \to E_n, \varphi_1= \p_{E_r} \circ \delta|_{M_X}$. Then $\varphi_1(z) = 0$, ${ D_z \varphi_1 = \Id}$ and $\varphi=\varphi_1|_{\p(M)} \circ \p.$ It follows that
\[\begin{array}{ccl}

\Phi_1&=& \lim\limits_{k \to \infty} \frac{1}{N_k} \sum_{j =0}^{N_k -1} \Lambda^{-j} \circ \varphi_1|_{\p(M)} \circ \pi \circ g^{\circ j} \circ \sigma|_{H(U) \setminus PC(f)}\\
&=&\lim\limits_{k \to \infty} \frac{1}{N_k} \sum_{j =0}^{N_k -1} \Lambda^{-j}  \circ \varphi_1|_{\p(M)} \circ h_j|_{H(U) \setminus PC(f)}
\end{array}\]

Note that $D_z h_j|_{H(U)} = \Lambda^j$, then we can deduce that $\Phi_1(z) = 0$ and $D_z \Phi_1=\Id$. Therefore, there exists a regular neighborhood $V$ of $z$ in $W$ such that $\Phi_1$ is biholomorphic on $V \cap H(U)$. Consequently, since $\sigma$ is a biholomorphism, the neighborhood $V$ induces an open neighborhood $\widetilde{V}=\sigma(V \setminus PC(f))$ in $\widetilde{W}$ such that $\Phi$ is biholomorphic on $\widetilde{V} \cap \M$. By shrinking $U$, we can assume that $V=U$ hence $\Phi|_{\widetilde{U} \cap M}$ is a biholomorphism.
\begin{figure}[H]
	\centering{
		\fbox{\resizebox{150mm}{!}{\includegraphics[scale=1]{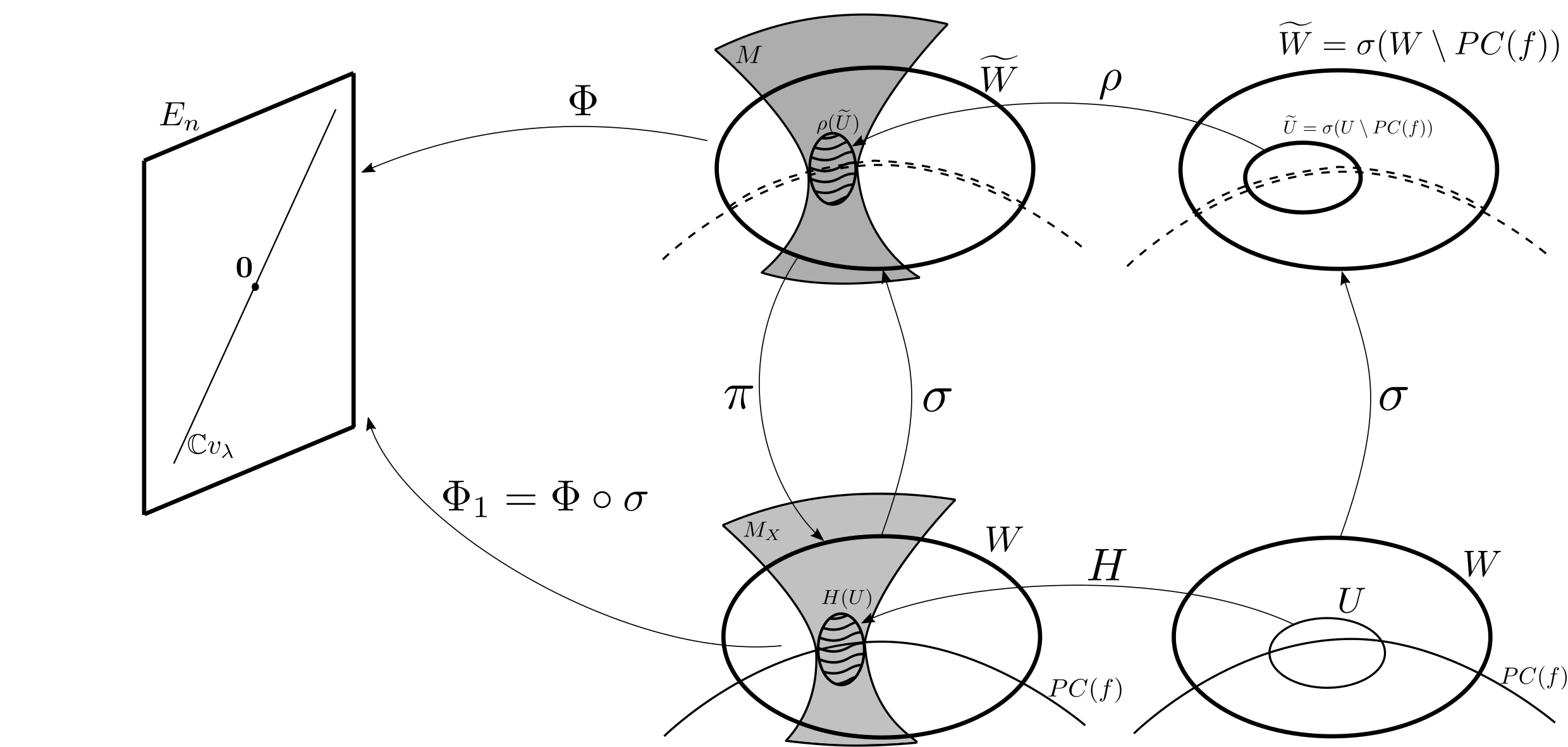}}}
		\caption{Constructions on $M$}
		
	}
\end{figure}

\subsection{Linearization along the neutral direction} The map $\Phi_1$ extends the image of $\Phi$ in the sense that $\Phi(M) \cup \Phi_1(H(U))$ contains a full neighborhood of $0$ in $E_n$. Let $v \in E_n$ be an eigenvector of $D_z f$ associated to $\lambda$. We will study $\Phi^{-1}(\C v)$ by studying $\Phi_1^{-1}(\C v)$.

Denote by $\Gamma_1$ the irreducible component of $\Phi_1^{-1}(\C v)$ containing $z$. Since $D_z \Phi_1 =\Id$, $\Gamma_1$ is a submanifold of dimension one of $H(U)$ near $z$ and $T_z \Gamma_1 =\C v$. Note that $v \notin T_z PC(f)$, then by shrinking $U$ if necessary, we can assume that $\Gamma_1 \cap PC(f) = \{z\}$. In other words, $\Gamma_1 \setminus \{z\}$ is a smooth component of $\Phi_1^{-1}(\C v)$ in $H(U) \setminus PC(f)$. 

Since $\Phi_1 = \Phi \circ \sigma|_{H(U) \setminus PC(f)}$ and $\sigma$ is a biholomorphism, there exists a unique irreducible component $\Gamma$ of $\Phi^{-1}(\C v)$ such that $\Gamma$ contains $\sigma(\Gamma_1 \setminus \{z\})$. Moreover, $\Phi(\Gamma)$ is a punctured neighborhood of $0$ in $\C v$. This means that $0 \notin \Phi(\Gamma)$ but $\Phi(\Gamma) \cup \{0\}$ contains an open neighborhood of $0$ in $\C v$. We will prove that $\Gamma$ is in fact biholomorphic to a punctured disc and $\Phi|_{\Gamma}$ is a biholomorphism conjugating $g|_{\Gamma}$ to the irrational rotation $\zeta \mapsto \lambda \zeta$. 

Following Section \ref{sec_rotation disc}, we consider $\Gamma_0 = \Gamma \setminus C_\Phi$ where $C_\Phi$ the set of critical points of $\Phi$. Then $\Gamma_0$ is a hyperbolic Riemann surface which is invariant by $g$. The map $g$ induces an automorphism $g|_{\Gamma_0}$ on $\Gamma_0$ such that $g|_{\Gamma_0}^{\circ j_k}$ converges to $\rho=\Id_M$ which is identity on $\Gamma_0$. 

On one hand, $\Gamma_0$ contains $\sigma(\Gamma_1 \setminus PC(f))$ hence $\Phi(\Gamma_0)$ is also a punctured neighborhood of $0$ in $\C v$. On another hand, $g$ restricted on $\sigma(\Gamma_1 \setminus PC(f))$ is conjugate to $h$ restricted on $\Gamma_1 \setminus PC(f)$. Note that $h$ fixes $z = \Gamma_1 \cap PC(f)$. Hence we can consider an abstract Riemann surface $\Gamma_0^{\star}=\Gamma_0 \cup \{z\}$ and two holomorphic maps $\iota: \Gamma_0 \to \Gamma_0^\star, \Phi^\star: \Gamma_0^\star \to \C v\subset E_n$ so that $\iota$ is an injective holomorphic map, $\Gamma_0^\star \setminus \iota(\Gamma_0) = \{z\}$, $\Phi^\star(z) = 0$ and the following diagram commutes.
\[\xymatrix{&\Gamma_0^\star \ar[dl]_{\Phi^\star}\\
	E_n & \Gamma_0 \ar[u]_{\iota} \ar[l]^{\Phi}}
\]
Moreover, $\Gamma_0^\star$ admits an automorphism $g^\star$ fixing the point $z$ with multiplier $\lambda$ and extends $g|_{\Gamma_0}$ in the sense that $g^\star \circ \iota = \iota \circ g$. Note that $\Phi^\star(\Gamma_0^\star) = \Phi(\Gamma_0) \cup \{0\}$ is bounded in $E_n$. Then by arguing as in Lemma \ref{lm_gamma 0}, we deduce that. 
\begin{lemma}
	The Riemann surface $\Gamma_0^\star$ is biholomorphic to a disc $\D(0,R), R \in (0,+\infty)$ and ${\Phi^\star: \Gamma_0^\star \to \Phi^\star(\Gamma_0^\star)=\D(0,R), \Phi^\star(z)=0}$ is a biholomorphism conjugating $g^\star$ to the irrational rotation $\zeta \to \lambda \zeta$.
\end{lemma}
Consequently, $\Gamma_0$ is biholomorphic to $\D(0,R)$ and $\Phi|_{\Gamma}$ is a biholomorphism. 
\begin{proposition}
	The set $\Gamma$ is smooth and the map\[
	\Phi|_{\Gamma}: \Gamma \to \Phi(\Gamma)=\D(0,R)^* \]
	is a biholomorphism with $R \in (0,+\infty)$.
\end{proposition}
\begin{proof}
	It is enough to prove that $\Gamma = \Gamma_0$. The idea is similar to the proof of Proposition \ref{prop_rotation disc}
	
	Note that $\Gamma_0$ is the complement of a discrete set $\Gamma \cap \Sing \Phi^{-1}(\C v)$ in $\Gamma$ and $\Gamma_0$ is biholomorphic to a punctured disc $\D(0,R)^*$. Moreover, $\Phi(\Gamma_0) \subset \Phi(\Gamma)$ is also a punctured neighborhood of $0$ and $\Gamma_0 \subset \Reg \Gamma$ has discrete complement.
	
	Then we can consider an abstract one dimensional analytic space $\Gamma^{\star}= \Gamma \cup \{z\}$ such that $\Gamma_0^\star \subset \Gamma^\star$ and $\Gamma^{\star} \setminus \Gamma_0^{\star}$ is a discrete set containing singular points of $\Gamma^{\star}$ (which is exactly $\Gamma \setminus \Gamma_0$). Then by arguing similarly to \ref{prop_rotation disc}, we can deduce that $\Gamma^{\star}$ is biholomorphic to $\D(0,R)$ hence the proposition is proved.
\end{proof}

\subsection{End of the proof}
Denote by $\tau_1:= \p \circ \left(\Phi|_{\Gamma}\right)^{-1} : \D(0,R)^* \to \C^n$. The map $\tau_1$ has a holomorphic extension to the map $\tau : \D(0,R) \to \C^n$ such that $\tau(0) = z, \tau'(0) =v$. The map $\tau$ takes values in $\p(M)$, which is bounded , hence the radial limit
\[\tau_\theta= \lim\limits_{r \rightarrow R^-} \tau(r e^{i \theta})
\]
exists for almost every $\theta \in [0, 2\pi)$.
\begin{proposition}
	$\tau_\theta \in PC(f)$ if it exists.
\end{proposition}
\begin{proof}
	Note that $\tau(\D(0,R)) \cap PC(f)= \left(\pi(\Gamma) \cup \Gamma_1\right) \cap PC(f) =\{z\}$. Hence $\tau(\D(0,R)\setminus \{0\}) \subset {M}$. It implies that $\tau(\D(0,R)\setminus \{0\}) \subset X$. Then by arguing similarly to Proposition \ref{radial limit}, we deduce that the radial limit $\tau_\theta \in PC(f)$ if this limit exists.
\end{proof} 

Recall that $Q$ is the defining polynomial of $PC(f)$ then $Q \circ \tau$ has vanishing radial limit for almost every $\theta \in [0,2 \pi)$. This means that $Q \circ \tau$ is identically $0$. Hence $\tau(\D(0,1)) \subset PC(f)$. This is a contradiction since $\tau'(0)= v \notin T_{z} PC(f)$. The proof the Proposition \ref{prop_transversal eigenvalue affine} is complete.
\section{Periodic cycles of post-critically algebraic endomorphisms of $\PP^2$}\label{sect_in singular part}
This section is devoted to the proof of Theorem \ref{PCA P2}.
\PCA*
If the periodic cycle is not in $PC(f)$, then $\lambda \neq 0$ and the result follows from Theorem \ref{first case}. Therefore, without loss of generality, we may assume that $z$ is a fixed point of $f$ in $PC(f)$. Note that if $\lambda$ is an eigenvalue of $f$ at a fixed point $z$, then $\lambda^j$ is an eigenvalue of $f^{\circ j}$ at the fixed point $z$. If we can prove that $\lambda^j$ is either superattracting or repelling, so is $\lambda$. Thus, in order to prove Theorem \ref{PCA P2}, we can always consider $f$ up to some iterates if necessary. 

After passing to an iterate, we may assume that the fixed point belongs to an invariant irreducible component $\Gamma$ of $PC(f)$. The reason why we have to restrict to dimension $n=2$ is that, in this case, $\Gamma$ is an algebraic curve. There is a normalization $\n:\hat \Gamma\to \Gamma$ where $\hat\Gamma$ is a smooth compact Riemann surface and $\n$ is a biholomorphism outside a finite set (see \cite{igor1994shafarevich}, \cite{gunning1990introduction} or \cite{chirka2012complex}). And there is a holomorphic endomorphism $\hat f:\hat\Gamma\to \hat\Gamma$ such that $\n\circ \hat f  = f\circ \n$. 

In section \ref{sec_invariant curve}, we analyse the dynamics of $\hat f:\hat \Gamma\to \hat{\Gamma}$ and in particular, we show that when $f$ is post-critically algebraic, then $\hat f$ is post-critically finite. In Section \ref{sec_fixed regular point}, we complete the proof in the case where the fixed point belongs to the regular part of $PC(f)$ and in Section \ref{sec_fixed singular point}, we complete the proof in the case where the fixed point belongs to the singular part of $PC(f)$. 
\subsection{Dynamics on an invariant curve}\label{sec_invariant curve} Assume that $f:\PP^2\to \PP^2$ is an endomorphism of degree $d\geq 2$ (not necessarily post-critically algebraic) and $\Gamma\subset \PP^2$ is an irreducible algebraic curve such that $f(\Gamma) = \Gamma$. Let $\n:\hat\Gamma\to \Gamma$ be a normalization of $\Gamma$ and $\hat f:\hat \Gamma\to \hat \Gamma$ be an endomorphism such that $\n\circ \hat  f = f\circ \n $. 

According to \cite[Theorem 7.4]{fornaess1994complex}, the endomorphism $\hat f : \hat\Gamma\to \hat{\Gamma}$ has degree $d \geq 2$. It follows from the Riemann-Hurwitz Formula that the compact Riemann surface $\hat \Gamma$ has genus $0$ or $1$. In addition, if the genus is $1$, then $\hat f$ has no critical point and all fixed points of $\hat f$ are repelling with common repelling eigenvalue $\lambda$ satisfying $|\lambda|= \sqrt{d'}$. If the genus is $0$, then the following lemma implies that the postcritical set of $\hat f$ and $f$ are closely related. 

\begin{lemma}
	Denote by $V_{\hat{f}}$ and $V_f$ the set of critical values of $\hat{f}$ and $f$ respectively. Then
	\[
	V_{\hat{f}}\subset \left\{ \begin{array}{lcr}
	\n^{-1}(V_f) \, &\mbox{if }& \, \Gamma \not\subset V_f\\
	\n^{-1}(\Sing V_f) \, &\mbox{if }& \, \Gamma\subset V_f
	\end{array}\right..
	\]
\end{lemma}
\begin{proof} The set of critical values of $\hat{f}$ is characterized by the following property: \textit{$x \notin V_{\f}$ if and only if for every $y \in \f^{-1}(x)$, $\f$ is injective near $y$.} Note that $\n: \hat{\Gamma} \rightarrow \Gamma$ induces a parametrization of the germ $(\Gamma,x)$ such that for every $x \in \Gamma$ and for every $y \in \n^{-1}(x)$, $\n$ is injective near $y$ (see also 2.3, \cite{wall2004singular}).
	\begin{itemize} 
		\item If $\Gamma \not\subset V_f$, let $x \notin \n^{-1}(V_f)$ and let $y \in \f^{-1}(x)$. Then $\n(y) \in f^{-1}( \n(x))$. Since $ \n(x) \notin V_f$ then $f$ is injective near $\n(y)$. Combining with the fact that $\n$ is locally injective, we deduce that $\f$ is injective near $y$. Thus $x \notin V_{\f}$.
		\item If $\Gamma \subset V_f$, let $x \notin \n^{-1}(\Sing V_f)$ and let $y \in \f^{-1}(x)$. Then $\n(y) \in f^{-1}(\n(x))$ and ${\n(x) \in \Reg V_f}$. By Proposition \ref{prop_locally biholomorphism in tangent direction}, we can deduce that $\n(y) \in f^{-1}(\Reg V_f)$ and $f|_{f^{-1}(\Reg V_f)}$ is locally injective. It implies that $\f$ is also injective near $y$. Hence $x \notin V_{\f}$.
	\end{itemize}
	Thus we obtain the conclusion of the lemma.
\end{proof}
\begin{proposition}\label{pcf}
	If $\hat{\Gamma}$ has genus $0$ and $f$ is a post-critically algebraic endomorphism then $\f$ is a PCF endomorphism.
\end{proposition}
\begin{proof}
	We have that
	\[PC(f) = \bigcup_{j\geq 1} V_{f^{\circ j}}\quad \text{and}\quad PC(\hat f) = \bigcup_{j\geq 1} V_{\hat f^{\circ j}}.\]
	Since $\n\circ \hat f^{\circ j} = f^{\circ j}\circ \n$ for all $j\geq 1$, applying the previous lemma to $f^{\circ j}$ and $\hat f^{\circ j}$ yields
	\[PC(\hat f)\subset \begin{cases}
	\n^{-1}\bigl(PC(f)\bigr)&\text{if } \Gamma\not\subset PC(f)\\
	\n^{-1}\bigl({\rm Sing}\, PC(f)\bigr)&\text{if } \Gamma\subset PC(f).
	\end{cases}\]
	In both cases, $PC(\hat f)$ is contained in the preimage by $\n$ of a proper algebraic subset of $\Gamma$, which therefore is finite. Since $\n$ is proper, $PC(\hat f)$ is finite and so, $\hat f$ is PCF. 
	
\end{proof}
Assume that $f$ has a fixed point $z$ which is a regular point of $\Gamma$ (which is not necessarily an irreducible component of $PC(f)$). Since $\n$ is a biholomorphism outside the preimage of singular points of $\Gamma$, the point $\n^{-1}(z)$ is a fixed point of $\hat{f}$ and $\n$ will conjugate $D_zf|_{T_z \Gamma}$ and $D_{\n^{-1}(z)} \hat{f}$. Denote by $\lambda$ the eigenvalue of $D_{z}f|_{T_z \Gamma}$. Then $\lambda$ is also the eigenvalue of $D_{\n^{-1}(z)}\hat{f}$. The previous discussion allows us to conclude that either $\lambda=0$ or $|\lambda|>1$. Thus we can deduce the following result.
\begin{lemma}\label{repel on normalisation}Let $f$ be a post-critically algebraic endomorphism of $\PP^2$ of degree $d\geq 2$, let $\Gamma\subset \PP^2$ be an invariant irreducible algebraic curve, let $z\in {\Reg} \Gamma$ be a fixed point of $f$ and let $\lambda$ be the eigenvalue of $D_zf|_{T_z\Gamma}$. Then, either $\lambda=0$ or $|\lambda|>1$.  
\end{lemma}

\subsection{Periodic cycles in the regular locus of the post-critical set}\label{sec_fixed regular point}
\begin{proof}[Proof of Theorem \ref{PCA P2} -- first part]
	Let $f$ be a post-critically algebraic endomorphism of $\PP^2$ with a fixed point $z$ which is a regular point of $PC(f)$. Denote by $\Gamma$ the irreducible component of $PC(f)$ containing $z$. Then $\Gamma$ is invariant by $f$. Denote by $\overline{D_{z} f}: T_{z} \PP^2 / T_{z} \Gamma \to T_{z} \PP^2 / T_{z} \Gamma$ the linear endomorphism induced by $D_{z} f$. Note that \[
	\ssp(D_{z} f) = \ssp(D_{z}f|_{T_{z} \Gamma})  \cup \ssp(\overline{ D_{z} f})\]
	By Proposition \ref{prop_locally biholomorphism in tangent direction}, the eigenvalue of $D_{z} f|_{T_{z} \Gamma}$ is not $0$ hence repelling by Lemma \ref{repel on normalisation}. By Proposition \ref{prop_transversal eigenvalue}, the eigenvalue of $\overline{D_{z} f}$ is either superattracting or repelling. Thus Theorem \ref{PCA P2} is proved when the fixed point is a regular point of the post-critical set. 	
\end{proof}
\subsection{Periodic cycles in the singular locus of the post-critical set}\label{sec_fixed singular point}
When the fixed point $z$ is a singular point of $PC(f)$, by passing to some iterates of $f$, we can assume that $f$ induces a holomorphic germ at $z$ which fixes a singular germ of curve at $z$ which is induced by some irreducible components of $PC(f)$. On the one hand, from the local point of view, there exists (in most of cases) a relation between two eigenvalues of $D_{z} f$ as a holomorphic germ fixing a singular germ of curve. On the other hand, from the global point of view, these eigenvalues can be identified with the eigenvalue the germ at a fixed point of the lifts of $f$ via the normalization of $PC(f)$. Then by Proposition \ref{pcf}, we can conclude Theorem \ref{PCA P2}.
\subsubsection{Holomorphic germ of $(\C^2,0)$ fixing a singular germ of curve}
Let $(\Sigma,0)$ be an irreducible germ of curve at $0$ in $(\C^2,0)$ defined by a holomorphic germ $g:(\C^2,0) \to (\C,0)$. In local coordinates $(x,y)$ of $\C^2$, if $g(0,y) \not\equiv 0$, i.e. $g$ does not identically vanish on $\{x=0\}$, it is well known that there exists an injective holomorphic germ $\gamma: (\C,0) \to (\C^2,0)$ of the form
\[
\gamma(t) = (t^m, \alpha t^n + O(t^{n+1}))
\]
parameterizing  $\Sigma$, i.e. $\gamma((\C,0)) = (\Sigma,0)$ (see \cite[Theorem 2.2.6]{wall2004singular}). If $\Sigma$ is singular, after a change of coordinates, $\alpha$ can be $1$ and $m$ and $n$ satisfy that $1< m< n, m\not| n$. The germ $\gamma$ is called a Puiseux parametrization. In fact, if $\Sigma$ is a germ induced by an algebraic curve $\Gamma$ in $\PP^2$, then $\gamma$ coincides with the germ induced by the normalization morphism. When $\Sigma$ is singular, the integers $m$ and $n$ are called the first two Puiseux characteristics of $\Gamma$ and they are invariants of the equisingularity class of $\Gamma$. In particular, $m$ and $n$ do not depend the choice of local coordinates. We refer to \cite{wall2004singular} for further discussion about Puiseux characteristics. We refer also to \cite{zariski2006probleme} and references therein for discussion about equisingular invariants. 

Now, consider a proper\footnote{Proper germ means that $g^{-1}(0) = 0$. In particular, an endomorphism of $\PP^2$ induces a proper germ at its fixed points. } holomorphic germ $g: (\C^2, 0) \to (\C^2,0)$ and a singular germ of curve $(\Sigma,0)$. If $\Sigma$ is invariant by $g$, i.e. $g(\Sigma) = \Sigma$, $g$ acts as a permutation on irreducible branches of $\Sigma$. Then by passing to some iterates of $g$, we assume that there exists an invariant branch. The following propositions show that there exists a relation between two eigenvalues of $D_{0} g$. 

When $g$ has an invariant singular branch, the following result was observed by Jonsson, \cite{jonsson1998some}.
\begin{proposition}\label{invariant cusp}
	Let $\Sigma$ be an irreducible singular germ of curve parametrized by  ${\gamma: (\C,0) \to (\C^2,0)}$ of the form
	\[\gamma(t) = (t^m, t^n + O(t^{n+1})), 1< m <n, m \not| n.\]
	Let $g: (\C^2,0) \to (\C^2,0)$ and $\hat{g}: (\C,0) \to (\C,0), \hat{g}(t) = \lambda t + O(t^2)$ be holomorphic germs such that \[g \circ \gamma = \gamma \circ \hat{g}.\]
	Then the eigenvalues of $D_{0}g$ are $\lambda^m$ and $\lambda^n$.
\end{proposition}  
\begin{proof}
	The germ $g$ has an expansion of the form
	\[g(x,y)=\big(ax+by+h_1(x,y),cx+dy+h_2(x,y)\big)
	\]
	where $h_1(x,y)=O(\|(x,y)\|^2), h_2(x,y)=O(\|(x,y)\|^2)$. Replacing those expansions in the equation $\gamma \circ \hat{g} = g \circ \gamma$, we have
	\begin{eqnarray*}
		\big(\lambda^m t^m +O(t^{m+1}), \lambda^n t^n +O(t^{n+1}) \big) &=&\big(a t^m +bt^n +h_1(t^m,t^n+O({t^{n+1}})),\\
		&& ct^m +dt^n +h_2(t^m,t^n+O(t^{n+1})
		)\big).
	\end{eqnarray*}
	Comparing coefficients of the term $t^m$ in each coordinate, we deduce that $a = \lambda^m$ and $c=0$. Comparing coefficients of the term $t^n$ in the second coordinate, since $m \nmid n$, the expansion of $h_2$ can not contribute any term of order $t^n$, hence $d=\lambda^n$. The linear part of $g$ has the form 
	$\left(\begin{array}{cc}
	a& b\\
	0&d 
	\end{array}\right)$
	hence $a,d$ are eigenvalues of $D_0 g$. In other words, $\lambda^m$ and $\lambda^n$ are eigenvalues of $D_{0} g$.
\end{proof}
When $g$ has an invariant smooth branch which is the image of another branch, $g$ is not an injective germ hence $0$ is an eigenvalue of $D_{0}g$. This case was not considered in \cite{jonsson1998some} since Jonsson assumed that there is no periodic critical point.
\begin{proposition}\label{Preperiodic regular}
	Let $g: (\C^2,0) \rightarrow (\C^2,0)$ be a proper holomorphic germ and let $\Sigma_1,\Sigma_2$ be irreducible germs of curves at $0$ such that $\Sigma_1 \neq \Sigma_2,g(\Sigma_1)=\Sigma_2, g(\Sigma_2)=\Sigma_2$. If $\Sigma_2$ is smooth then the eigenvalues of $D_{z} g$ are $0$ and $\lambda$ where $\lambda$ is the eigenvalue of $D_{0} g|_{T_{0} \Sigma_2}$.
\end{proposition}
\begin{proof} Since $\Sigma_2$ is smooth, we choose a local coordinates $(x,y)$ of $(\C^2,0)$ such that ${\Sigma_2 = \{x=0\}}$. Since $\Sigma_1$ and $\Sigma_2$ are distinct irreducible germs, the defining function of $\Sigma_1$ does not identically vanish on $\Sigma_2$. Then we can find a Puiseux parametrization of $\Sigma_1$ of the form
	\[\gamma(t)=(t^m,\alpha t^n+O(t^{n+1})), \alpha \in \C\setminus \{0\}
	\]
	where $m,n$ are positive integers (see \cite[Theorem 2.2.6]{wall2004singular}). The germ $g$ has an expansion of the form
	\[g(x,y)=(ax+by+h_1(x,y),cx+dy+h_2(x,y))
	\]
	where $h_1(x,y)=O(\|(x,y)\|^2), h_2(x,y)=O(\|(x,y)\|^2)$. The invariance of $\Sigma_2$ implies that $b=0$ and $g$ has the form
	\[g(x,y) = (x(a+h_3(x,y)), cx+dy+h_2(x,y))
	\]
	where $h_3(x,y)=O(\|(x,y)\|),h_2(x,y)=O(\|(x,y)\|^2)$. Replacing $\gamma$ and $f$ in the equation ${g(\Sigma_1)=\Sigma_2}$, we have
	\[t^m(a+h_3(t^m,\alpha t^n+O(t^{n+1})))=0
	\]
	hence $a=0$. Then the linear part of $D_{0}g$ is $(0,cx+dy)$ hence $0$ and $d$ are the eigenvalues of $D_{0} g$ where $d$ is the eigenvalue of the restriction of $D_{0}g$ to $T_{0} \Sigma_2$. 
\end{proof}
And finally, if $g$ has two invariant smooth branches which are tangent, we have the following result (see also \cite{jonsson1998some}).
\begin{proposition}\label{tangential regular}
	Let $g: (\C^2,0) \rightarrow (\C^2,0)$ be a proper holomorphic germ and let $\Sigma_1,\Sigma_2$ be irreducible invariant germs of smooth curves at $0$. If $\Sigma_1$ and $\Sigma_2$ intersect tangentially, i.e. $\Sigma_1 \neq \Sigma_2$ and ${T_{0} \Sigma_1 = T_{0} \Sigma_2 }$ then there exists a positive integer $m$ such that the eigenvalues of $D_{0} g$ are $\lambda$ and $\lambda^m$ where $\lambda$ is the eigenvalue of $D_{0}g|_{T_{0}\Sigma_1}.$ \end{proposition}
\begin{proof}
	Since $\Sigma_1$ is smooth, we can choose a local coordinates $(x,y)$ such that $\Sigma_1=\{y=0\}$. The defining function of $\Sigma_2$ can not identically vanish on $\{x=0\}$ since otherwise, $\Sigma_2$ and $\Sigma_1$ would not tangent. Then $\Sigma_2$ has a parametrization of the form
	\[\gamma(t)=(t,  t^m+O(t^{m+1})).
	\]
	Since $\Sigma_1=\{y=0\}$ is invariant, $g$ has an expansion in the coordinates $(x,y)$ of the form:
	\[g(x,y)=\big(\lambda x+by+h_1(x,y),y(d+h_2(x,y))\big)
	\]
	where $\lambda,b,d \in \C, h_1(x,y)=O(\|(x,y)\|^2),h_2(x,y)=O(\|(x,y)\|)$. The linear part of $D_{\0}g$ is ${(\lambda x+by,dy)}$ thus $\lambda,d$ are eigenvalues of $D_{z}f$. Letting $x=t,y = t^m + O(t^{m+1})$, we have
	\[
	f(t,t^m + O(t^{m+1})) = (\lambda t + O(t^2), d t^m + O(t^{m+1} ))
	\]Since $f(\Sigma_2) = \Sigma_2$, we deduce that $d=\lambda^m$. Note that $\lambda$ is the eigenvalue of $D_{\0} g|_{T_{\0} \Sigma_1}$.
\end{proof}
Using these observations, we can conclude the proof of Theorem \ref{PCA P2}.
\begin{proof}[Proof of Theorem \ref{PCA P2} -- final]
	Let $f$ be a post-critically algebraic endomorphism of $\PP^2$ and let $z$ be a fixed point such that $z$ is a singular point of $PC(f)$. We will look at the germ $(PC(f),z)$ induced by $PC(f)$ at $z$ and prove Theorem \ref{PCA P2} depending on how $f$ acts on irreducible branches of $(PC(f),z)$. Passing to an iterate of $f$ if necessary, we assume that there exists a branch $\Sigma$ of $(PC(f),z)$ such that $f(\Sigma) = \Sigma$. Denote by $\Gamma$ the irreducible component of $PC(f)$ inducing $\Sigma$. Then $\Gamma$ is also invariant by $f$. Denote by $\n: \hat{\Gamma} \to \Gamma$ the normalization of $\Gamma$ and by $\f$ the lifting of $f$ by the normalization $\n: \hat{\Gamma} \rightarrow \Gamma$.
	
	If $\Sigma$ is singular then $\n^{-1}(z)$ is a finite set and $\f(\n^{-1}(z)) \subset \n^{-1}(z)$ since $z$ is a fixed point of $f$. Then by passing up to some iterations, we can assume that $\f$ fixes a point $w_0 \in \n^{-1}(z)$. By Proposition \ref{invariant cusp}, the eigenvalues of $D_{z} f$ are $\lambda^m, \lambda^n$ where $\lambda$ is the eigenvalue of $D_{w_0} \f$ and $m$ and $n$ are the first two Puiseux characteristics of $\Sigma$. By Proposition \ref{pcf}, $\lambda$ is either superattracting or repelling. Hence so are $\lambda^m$ and $\lambda^n$.
	
	If $\Sigma$ is smooth, the tangent space $T_{z} \Sigma$ is well-defined and invariant by $D_z f$. Denote by $\lambda$ the eigenvalue of $D_{z} f|_{T_z \Sigma}$. Even if $\Gamma$ can be singular (for example, $z$ can be a self-intersection point of $\Gamma$), there exists a point $w \in \hat{\Gamma}, \n(w) = z$ such that $w$ is a fixed point of $\hat{f}$ and $\n$ induces an invertible germ $\n: (\hat{\Gamma},w) \to (\Sigma,z)$. Arguing similarly to Lemma \ref{repel on normalisation}, we can deduce that either $\lambda=0$ or $|\lambda |>1$. To deal with the other eigenvalue, since $(PC(f), z)$ is singular, we have one of the following cases: 
	\begin{itemize}
		\item[1.] There exists a branch $\Sigma_1$ such that $f(\Sigma_1) = \Sigma$. By Proposition \ref{Preperiodic regular}, the eigenvalues of $D_z f$ are $0$ and $\lambda$ where $\lambda$ is the eigenvalue of $D_{z} f|_{T_{z} \Sigma}$. Hence we are done by previous discussion.
		\item[2.] There exists a smooth invariant branch $\Sigma_1$ such that $\Sigma$ and $\Sigma_1$ intersect transversally. Denote by $\Gamma_1$ the irreducible component of $PC(f)$ containing $\Sigma_1$ (which can be $\Gamma$) thus $\Gamma_1$ are invariant by $f$. Since $\Sigma $ and $\Sigma_1$ are transversal, the eigenvalues of $D_{z}f$ are $\lambda$ and $\lambda_1$ where $\lambda_1$ is the eigenvalue of $D_{z} f|_{T_{z} \Sigma_1}$. Arguing similarly to the case of $\lambda$, we deduce $\lambda_1$ is also either superattracting or repelling.
		\item[3.] There exists a smooth invariant branch $\Sigma_1$ such that $\Sigma$ and $\Sigma_1$ intersect tangentially. Denote by $\Gamma_1$ the irreducible component of $PC(f)$ inducing $\Sigma_1$ (which is possibly equal to $\Gamma$) thus $\Gamma_1$ are invariant by $f$. By Proposition \ref{tangential regular}, the eigenvalues of $D_{z} f$ are $\lambda$ and $\lambda^m$. Note that $\lambda$ is either superattracting or repelling hence so does $\lambda^m$.
	\end{itemize}\qedhere
	
\end{proof}
\bibliographystyle{alpha}
\bibliography{ref}
\end{document}